\documentclass[a4paper, 12pt]{amsart}

\usepackage{enumerate}
\usepackage{mathrsfs}
\usepackage{amssymb}
\usepackage[all]{xy}
\usepackage{MnSymbol}
\usepackage{nicefrac}
\usepackage{tikz}
\usetikzlibrary{matrix,arrows,decorations.pathmorphing}

\title[Branching laws and Hermitian symmetric spaces]
{Symplectic branching laws and Hermitian symmetric spaces}
\author[Benjamin Schwarz and Henrik Sepp\"anen]{Benjamin Schwarz and Henrik Sepp\"anen*} 

\thanks{*supported by the DFG priority program 1388 "Representation 
Theory"}

\keywords{branching law, holomorphic line bundle, Hermitian symmetric space, 
moment map, Jordan pair, Okounkov body}
\date{\today}

\address{Benjamin Schwarz,Universit\"{a}t Paderborn,
Fakult\"{a}t f\"{u}r Elektrotechnik, Informatik und Mathematik,
Institut f\"{u}r Mathematik, Warburger Str. 100,
33098 Paderborn}
\email{bschwarz@math.upb.de}

\address{Henrik Sepp\"{a}nen,Universit\"{a}t Paderborn,
Fakult\"{a}t f\"{u}r Elektrotechnik, Informatik und Mathematik,
Institut f\"{u}r Mathematik, Warburger Str. 100,
33098 Paderborn}
\email{henriksp@math.upb.de}

\newcommand{\C}{\mathbb{C}}
\newcommand{\R}{\mathbb{R}}
\newcommand{\N}{\mathbb{N}}
\newcommand{\Z}{\mathbb{Z}}

\newcommand{\g}{\mathfrak{g}}
\newcommand{\fu}{\mathfrak{u}}
\newcommand{\fk}{\mathfrak{k}}
\newcommand{\ft}{\mathfrak{t}}
\newcommand{\fh}{\mathfrak{h}}
\newcommand{\fp}{\mathfrak{p}}

\newcommand{\fb}{\mathfrak{b}}

\newcommand{\fl}{\mathfrak{l}}
\newcommand{\fn}{\mathfrak{n}}
\newcommand{\fq}{\mathfrak{q}}

\newcommand{\linebundle}{\mathscr{L}}

\newcommand{\CC}{\mathbb{C}}
\newcommand{\RR}{\mathbb{R}}
\newcommand{\NN}{\mathbb{N}}
\DeclareMathOperator{\Aut}{Aut}
\DeclareMathOperator{\End}{End}
\DeclareMathOperator{\GL}{GL}
\DeclareMathOperator{\Hom}{Hom}
\DeclareMathOperator{\Id}{Id}
\DeclareMathOperator{\Tr}{Tr}
\DeclareMathOperator{\rank}{rk}
\DeclareMathOperator{\Gr}{Gr}
\DeclareMathOperator{\Ad}{Ad}

\newcommand{\mf}[1]{\mathfrak{#1}}
\newcommand{\JTP}[3]{\left\{#1,\,#2,\,#3\right\}}		
\newcommand{\B}[2]{B_{#1,\,#2}}											
\newcommand{\trans}[1]{t_{#1}}											
\newcommand{\qtrans}[1]{\tilde t_{#1}}							
\newcommand{\Set}[2]{\left\{#1\,\middle|\,#2\right\}} 
\newcommand{\set}[2]{\{#1\,|\,#2\}}								  
\newcommand{\Tri}{S}																
\newcommand{\Peirce}{\mathcal P}										
\newcommand{\Idem}{\mathcal I}											
\newcommand{\GP}[2]{{\left[#1:#2\right]}}						
\renewcommand{\b}[1]{{\bf #1}}
\newcommand{\sH}{\mathcal H}												
\DeclareMathOperator{\proj}{pr}											
\newcommand{\half}{{\nicefrac{1}{2}}}								
\newcommand{\JP}[1]{\text{\sffamily JP#1}}					
\newcommand{\op}{{\textup{op}}}

\newtheorem{prop}{Proposition}[section]
\newtheorem{lemma}[prop]{Lemma}
\newtheorem{cor}[prop]{Corollary}
\newtheorem{thm}[prop]{Theorem}

\theoremstyle{definition}

\newtheorem{rem}[prop]{Remark}

\begin{document}

\begin{abstract}
Let $G$ be a complex simple Lie group, and let $U \subseteq G$ be 
a maximal compact subgroup. Assume that $G$ admits a homogenous space
$X=G/Q=U/K$ which is a compact Hermitian symmetric space. 
Let $\mathscr{L} \rightarrow X$ be the ample line bundle which generates
the Picard group of $X$. In this paper we study the restrictions 
to $K$ of the family $(H^0(X, \mathscr{L}^k))_{k \in \N}$ of irreducible 
$G$-representations. We describe explicitly the moment polytopes for 
the moment maps $X \rightarrow \fk^*$ associated to positive
integer multiples of the Kostant-Kirillov symplectic form on $X$, 
and we use these, together with an explicit 
characterization of the closed $K^\C$-orbits 
on $X$, to find the decompositions of the spaces
$H^0(X,\mathscr{L}^k)$. 
We also construct a natural Okounkov body for $\mathscr{L}$ and 
the $K$-action, and identify it with the smallest of 
the moment polytopes above. 
In particular, the Okounkov body is a convex polytope. In fact, 
we even prove the stronger property that the semigroup defining 
the Okounkov body is finitely generated.  
\end{abstract}

\maketitle

\section{Introduction}

In this paper we consider the following setting. Let $G$ be a 
complex simple Lie group, and assume that $G$ admits a quotient 
$X:=G/Q$ which is a compact Hermitian symmetric space. Then 
$Q$ is a maximal parabolic subgroup. Moreover, we can 
write $X$ as $X=U/K$, where $U \subseteq G$ is a maximal compact 
subgroup, and $K:=U \cap Q$. The Picard group
of $X$, which is isomorphic to the group of 
holomorphic characters $Q \rightarrow \C^\times$, is 
$\Z$. Let $\mathscr{L} \rightarrow X$ be the ample 
generator for the Picard group. 
We are concerned with the decomposition under $K$ of the
irreducible $G$-representations given by the family 
$H^0(X, \mathscr{L}^k)$, where $k \in \N$. 

In order to put our approach to the decomposition 
problem into its proper framework we make 
a small digression into a more general setting. For 
a more thorough treatment we refer to \cite{S95} 
and the references therein.
We now temporarily let $K$ denote an arbitrary
compact Lie group, and assume that $K$ acts 
holomorphically and in a Hamiltonian fashion 
on the connected compact K\"ahler manifold 
$(M, \omega)$ with moment map
$\tau \rightarrow \fk^*$. Assume that $(M, \omega)$ 
admits a prequantum line bundle $\mathcal{L}$. 
The action of $K$ then lifts to an action on 
$\mathcal{L}$, and hence we have 
a representation of $K$ on the space 
$H^0(M, \mathcal{L})$. One may then 
ask how $H^0(M, \mathcal{L})$ decomposes 
under $K$. For this purpose it is useful 
to realize the irreducible $K$-representation 
of highest weight $\xi$ as the space 
of holomorphic sections 
$H^0(\mathcal{O}^K_\xi, \mathscr{L}_\xi)$, where
$\mathcal{O}^K_\xi \subseteq \fk^*$ is the 
coadjoint orbit through $\xi$, and 
$\mathscr{L}_\xi$ is the prequantum line bundle 
attached to the Kostant-Kirillov symplectic form 
on $\mathcal{O}^K_\xi$. 
It is well-known that the multiplicity 
of the representation $H^0(\mathcal{O}^K_\xi, \mathscr{L}_\xi)$
in $H^0(M, \mathcal{L})$ is given by 
the dimension of the space, 
$H^0(M \times \overline{\mathcal{O}^K_\xi}, 
\mathcal{L} \boxtimes \overline{\mathscr{L}_\xi^*})^K$, 
of $K$-invariant holomorphic sections of 
the line bundle 
$\mathcal{L} \boxtimes \overline{\mathscr{L}_\xi^*}
\rightarrow M \times \overline{\mathcal{O}^K_\xi}$. 
Here $\overline{\mathcal{O}_\xi^K}$ denotes 
the topological space $\mathcal{O}_\xi^K$ equipped
with the reverse complex structure and the symplectic 
form given by $-1$ times the Kostant-Kirillov 
form. The fibre over $x \in \mathcal{O}^K_\xi$ of the line bundle 
$\overline{\mathscr{L}_\xi^*}$ 
consists of the space of antilinear complex-valued
functionals on $(\mathscr{L}_\xi)_x$ (cf. \cite{GS82}).
The group $K$ now acts holomorphically and in a Hamiltonian 
fashion on $M \times \overline{\mathcal{O}^K_\xi}$ with 
moment map 
$\tau^\xi:M \times \overline{\mathcal{O}^K_\xi} \rightarrow \fk^*$
given by $\tau^\xi(m,f):=\tau(m)-f$. 

An obvious question is whether the space 
$H^0(M \times \overline{\mathcal{O}^K_\xi}, 
\mathcal{L} \boxtimes \overline{\mathscr{L}_\xi^*})^K$
can be interpreted as the space of all holomorphic sections
of some line bundle over some ``quotient'' of 
$M \times \overline{\mathcal{O}^K_\xi}$ by $K$. Indeed, this 
holds for the Mumford quotient 
$$(M \times \overline{\mathcal{O}^K_\xi})_0:=
(M \times \overline{\mathcal{O}^K_\xi})_{ss}//K^\C,$$
where 
$(M \times \overline{\mathcal{O}^K_\xi})_{ss}$ is the open 
subset consisting of the \emph{semistable points} of 
$(M \times \overline{\mathcal{O}^K_\xi})$. 
An interesting feature, and one which links the Mumford quotient 
to the symplectic geometry, is that  
$(M \times \overline{\mathcal{O}^K_\xi})_0$ is
homeomorphic to the topological quotient $(\tau^\xi)^{-1}(0)$
(cf. \cite[Thm. 2.5]{S95}), the \emph{symplectic reduction} 
at $0$.
Moreover, the following result holds.
\begin{thm}\label{T: sjamaar}(\cite[Corollary 1]{S95}) 
If $\xi$ does not lie in the image $\tau(M)$, then 
the irreducible representation of highest weight 
$\xi$ does not occur in $H^0(M, \mathcal{L})$.
\end{thm}

The space $(M \times \overline{\mathcal{O}^K_\xi})_0$ also
carries the structure of a complex projective variety. In fact, 
it is isomorphic to 
$\mbox{Proj}(\bigoplus_{k=0}^{\infty} 
H^0(M \times \overline{\mathcal{O}^K_\xi}, 
(\mathcal{L} \boxtimes \overline{\mathscr{L}_\xi^*})^k)^K)$. 
Moreover, for big enough $q$, the line bundle 
$\mathcal{L}^q \boxtimes (\overline{\mathscr{L}_\xi^*})^q$ 
induces a line bundle
$(\mathcal{L}^q \boxtimes (\overline{\mathscr{L}_\xi^*})^q)_0$ 
over $(M \times \overline{\mathcal{O}^K_\xi})_{ss}//K^\C$, 
the total space of which is 
$((\mathcal{L} \boxtimes \overline{\mathscr{L}_\xi^*})^q)
\mid_{(M \times \overline{\mathcal{O}^K_\xi})_{ss}})/K^\C$. 
For such $q$, there is an isomorphism 
\begin{equation*}
H^0(M \times \overline{\mathcal{O}^K_\xi}, 
(\mathcal{L} \boxtimes \overline{\mathscr{L}_\xi^*})^q)^K
\cong 
H^0((M \times \overline{\mathcal{O}^K_\xi})_0,
(\mathcal{L}^q \boxtimes (\overline{\mathscr{L}_\xi^*})^q)_0).
\end{equation*}
Under favourable conditions, \emph{e.g.}\ the vanishing 
of all cohomology groups
$H^i((M \times \overline{\mathcal{O}^K_\xi})_0,
(\mathcal{L}^q \boxtimes (\overline{\mathscr{L}_\xi^*})^q)_0))$, 
for $i >0$, the asymptotics of 
$\mbox{dim} H^0((M \times \overline{\mathcal{O}^K_\xi})_0,
(\mathcal{L}^{kq} \boxtimes (\overline{\mathscr{L}_\xi^*})^{kq})_0)$
as $k \rightarrow \infty$ is given by the Riemann-Roch theorem
for (singular) complete schemes (cf. \cite[Example 18.3.6]{Fu98}).

Although this machinery works well in principle, a major 
obstruction for applying it in particular cases in order 
to obtain explicit asymptotic expressions for multiplicities
is that the moment maps $\tau^\xi$, and hence their fibres, 
are in general notoriously hard to compute. 

In this paper we are able to compute the moment 
map $\mu_\fk: X \rightarrow \fk^*$ for the $K$-action
by using an explicit Jordan-theoretic
description of $X$. We also prove that 
for any $\nu \in \mu_\fk(X)$, the stabilizer,
$K_\nu$, of $\nu$ acts transitively on the fibre 
$\mu_\fk^{-1}(\nu)$. As a consequence, the 
decomposition of $H^0(X, \mathscr{L}^k)$ under 
$K$ is multiplicity free for every $k \in \N$. 
We also describe explicitly the moment polytope for 
$k\mu_\fk$ for any $k \in \N$, i.e., the intersection 
of $k\mu_\fk(X)$ with a closed Weyl chamber, as well 
as the integral points in the moment polytope.
By Theorem \ref{T: sjamaar}, these are the only weights 
that can occur as highest weights of irreducible 
$K$-representations in $H^0(X, \mathscr{L}^k)$. 
We prove that all these integral points in fact 
do occur. In fact, from the particular form of the integral 
weights in the moment polytopes for the $k\mu_\fk$ it 
turns out that it suffices to prove this for $k=1$, i.e., 
that all the integral points in the moment 
polytope for $\mu_\fk$ parameterize 
irreducible $K$-representations in $H^0(X, \mathscr{L})$. 

In the special case $k=1$ we prove that the 
integral points stand in a one-to-one correspondence
with the closed $K^\C$-orbits, $X_0,\,\ldots,\,X_r$,
in $X$. The number $r$ is the \emph{rank} of $X$ as a symmetric
space. We also give a Jordan-theoretic 
characterization of these orbits. 
Using this characterization, we give a geometric 
decomposition of $H^0(X, \mathscr{L})$ under $K$. The 
$K$-equivariant embedding $W_i \rightarrow H^0(X, \mathscr{L})$ 
of the irreducible representation $W_i$ corresponding to 
the orbit $X_i$ is a section for the restriction map 
$H^0(X, \mathscr{L}) \rightarrow H^0(X_i, \mathscr{L}\mid_{X_i})$.

We also define an Okounkov body for the line bundle 
$\mathscr{L}$ and the $K$-action (cf. \cite{Ok}) by using 
a canonical local trivialization of sections. 
The semigroup defining the Okounkov body 
describes the initial monomial terms  
of the polynomials that are local trivializations
of highest weight vectors for irreducible 
$K$-subrepresentations. Using the decompositions 
for the spaces $H^0(X, \mathscr{L}^k)$ under $K$ 
we are able to prove that this semigroup is 
finitely generated.

The paper is organized as follows. In Section 2 and Section 3 
we recall the preliminaries from Lie theory and Jordan theory, 
respectively. In Section 4 we study the local trivializations
of holomorphic sections and use these to define an Okounkov 
body for $\mathscr{L}$ and the $K$-action. 
Section 5 is devoted to closed $K^\C$-orbits in $X$ 
and a geometric decomposition theorem. In Section 6 we 
describe the moment polytope and the symplectic reductions. 
In Section 7 we prove the decomposition theorem 
for $H^0(X, \mathscr{L}^k)$ using the results from 
previous sections. In Section 8 use the results from 
Section 7 to identify the Okounkov body with the 
moment polytope for $\mu_\fk$.

{\bf Acknowledgement:} We would like to thank Joachim 
Hilgert and Harald Upmeier for stimulating discussions on 
these topics.

\section{Preliminaries from Lie theory}\label{sec:PerlimLieTheo}
Let $\g$ be a complex simple Lie algebra, 
and let $\mathfrak{h}$ be a fixed Cartan 
subalgebra of $\g$. Let $\Phi=\Phi(\g,\mathfrak{h})$ 
be the set of roots with respect to $\mathfrak{h}$. 
Let $\Phi^+ \subseteq \Phi$ be a positive system, and 
$\{\beta_1,\ldots, \beta_s\}$ the corresponding simple roots. 
For every $\alpha \in \Phi^+$, fix an 
$\mathfrak{sl}_2$-triple $\{E_\alpha, H_\alpha, E_{-\alpha}\}$ 
with $H_\alpha \in \mathfrak{h}$, $E_\alpha \in \g_\alpha$, 
and $E_\alpha \in \g_{-\alpha}$, normalized so that the 
identitiy
\begin{equation*}
[E_\alpha, E_\beta]=N_{\alpha, \beta}E_{\alpha+\beta}
\end{equation*}
holds for all $\alpha, \beta \in \Phi$, 
with constants $N_{\alpha, \beta} \in \R$ satisfying
$N_{-\alpha, -\beta}=-N_{\alpha, \beta}$. 
Also, set $F_\alpha:=E_{-\alpha}$ for $\alpha \in \Phi^+$. 
Let 
\begin{equation*}
\fu:=\bigoplus_{j=1}^s iH_{\beta_j}
\oplus \bigoplus_{\alpha \in \Phi^+} \R(E_\alpha-F_\alpha) 
\oplus \bigoplus_{\alpha \in \Phi^+} \R i(E_\alpha+F_\alpha)
\end{equation*}
be the canonical compact real form of $\g$, and 
$\theta: \g \rightarrow \g$
be the associated Cartan involution of $\g$.
Let 
\begin{equation*}
\ft:=\bigoplus_{j=1}^s iH_{\beta_j} \subseteq \fu
\end{equation*} 
be the maximal abelian subalgebra 
of $\fu$ with $\fh^\theta=\ft$. Define
\begin{align*}
\mathfrak{n}^+:=\bigoplus_{\alpha \in \Phi^+ } \g_\alpha\,,\quad
\mathfrak{n}^-:=\bigoplus_{\alpha \in -\Phi^+ } \g_\alpha\,,
\end{align*} 
and let 
\begin{equation*}
\mathfrak{b}:=\fh \oplus \fn^+
\end{equation*} 
be the Borel subalgebra defined by the positive system $\Phi^+$.
Then $\g$ can be decomposed as
\begin{equation*}
\g=\mathfrak{b} \oplus \mathfrak{n}^-.
\end{equation*}
For a root, $\alpha$, let $m_{\beta_i}(\alpha)$
be the multiplicity of $\beta_i$ in $\alpha$, i.e., 
$\alpha=\sum_{i=1}^s m_{\beta_i}(\alpha)\beta_i$.

In this paper we shall be concerned with the 
special case when $\g$ admits a simple root that 
has multiplicity at most one in every positive root. 
Assume therefore that $\beta_1$ is such a simple root.
We define subsets of $\Phi$ by
\begin{equation*}
\Phi_Q:=\{\alpha \in \Phi \mid m_{\beta_1}(\alpha) \geq 0\},
\quad \Phi_L:=\{\alpha \in \Phi \mid m_{\beta_1}(\alpha)=0\},  
\end{equation*}
and Lie subalgebras
\begin{align*}
\fl := \fh \oplus \bigoplus_{\alpha \in \Phi_L} \g_\alpha\,,\quad
\fq := \fh \oplus \bigoplus_{\alpha \in \Phi_Q} \g_\alpha\,.
\end{align*}
We also define 
\begin{align*}
\fp^+:=\bigoplus_{\alpha \in \Phi_Q \setminus \Phi_L} \g_\alpha\,,\quad
\fp^-:=\bigoplus_{\alpha \in \Phi_Q \setminus \Phi_L} \g_{-\alpha}\,.
\end{align*}
Then we have
\begin{equation*}
\fq=\fl \oplus \fp^+.
\end{equation*}
The Lie algebra $\fq$ is a maximal parabolic subalgebra of $\g$ 
containing $\fb$, and $\fp^+$ and $\fl$ are the nilpotent radical 
of $\fq$ and the Levi subalgebra of $\fq$, respectively.
From the assumption that 
$m_{\beta_1}(\alpha) \in \{-1,0,1\}$ for any root $\alpha$ it
readily follows that $\fp^+$ is an abelian subalgebra. 

The Lie algebra $\fl$ is reductive with semisimple 
part $$\fl':=[\fl, \fl]=\fh_L \oplus 
\bigoplus_{\alpha \in \Phi_L} \g_\alpha,$$ 
where $$\fh_L:=H_{\beta_2} \oplus \cdots \oplus H_{\beta_s}$$
is a Cartan subalgebra of $\fl'$. Then 
$\Phi_L$ is the set of roots of $\fl$ with respect 
to $\fh_L$. The set $\Phi_L^+:=\Phi_L \cap \Phi^+$ is a positive
system in $\Phi_L$, and we accordingly 
define the subalgebras
\begin{equation*}
\fn_L^+:=\bigoplus_{\alpha \in \Phi_L^+} \fl_\alpha\,,\quad
\fn_L^-:=\bigoplus_{\alpha \in -\Phi_L^+} \fl_\alpha
\end{equation*}
of $\fl$. The subalgebra $\fl$ is obviously $\theta$-invariant, 
and we set 
\begin{equation*}
\fk:=\fl^\theta=\fl \cap \fu.
\end{equation*}
Let $\zeta_0 \in \fk$ be a basis vector for 
$\mathfrak{z(\fl)}$, the centre of $\fl$. 
Then $$\fl=\fl' \oplus \C \zeta_0.$$

We now return to the roots in $\Phi^+$ and construct 
a particular numbering of them. Equip the root lattice 
with the lexicographic order coming from the 
identification with the lattice 
$\Z \beta_1 \oplus \cdots \oplus \Z \beta_s$.
Then $\beta_1 >\ldots > \beta_s$. Moreover, 
$\alpha > \gamma$, for any 
$\alpha \in \Phi_Q\setminus\Phi_L$ 
and $\gamma \in \Phi_L$.

We construct a maximal
set $\{\gamma_1,\ldots, \gamma_r\} \subseteq \Phi_Q \setminus \Phi_L$ 
of \emph{strongly orthogonal roots}, i.e., the $\gamma_i$ satisfy
the property that for all pairs $\{\gamma_i,\gamma_j\}$ is 
neither $\gamma_i+\gamma_j$, nor $\gamma_i-\gamma_j$ a root.

\noindent First, put $\gamma_1:=\beta_1$. Assuming that 
$\gamma_1,\ldots, \gamma_i$ are defined, let 
$\gamma_{i+1}$ be the smallest root in 
$\Phi_Q \setminus \Phi_L$ such that 
$\gamma_j \pm \gamma_{i+1} \notin \Phi$ for $j=1,\ldots, i$.

Given the roots $\gamma_1, \ldots, \gamma_r$, we 
now consider a particular decomposition of the 
the maximal abelian subalgebra $\ft \subseteq \fk$. 
Let $$\kappa: \g \times \g \rightarrow \C$$ 
be the Killing form of $\g$. 
Put $\mathfrak{s}':=\R iH_{\gamma_1} \oplus \cdots
\oplus \R iH_{\gamma_r} \subseteq \ft$, and 
consider the decomposition 
$\ft=\mathfrak{s}' \oplus (\mathfrak{s}')^{\perp}$, 
where $(\mathfrak{s}')^\perp \subseteq \ft$ is the 
orthogonal complement to $\mathfrak{s}'$ with 
respect to the restriction of $\kappa$ to $\ft$. 
Notice that
\begin{equation*}
(\mathfrak{s}')^{\perp}=
\{H \in \ft \mid \gamma_i(H)=0, \quad i=1,\ldots, r\}.
\end{equation*}

With respect to this decomposition of $\ft$, let
$\zeta \in \ft$ denote the $(\mathfrak{s}')^\perp$-
component of $\zeta_0$. 
Now put $\mathfrak{s}:=\mathfrak{s}' \oplus \R \zeta$, 
and consider the decomposition 
\begin{equation}\label{eq:MaxAbDecomp}
\ft=\mathfrak{s} \oplus (\mathfrak{s})^{\perp},
\end{equation}
where $(\mathfrak{s})^\perp \subseteq \ft$ is the 
orthogonal complement to $\mathfrak{s}$ with 
respect to the restriction of $\kappa$ to $\ft$.

We now enumerate the roots in 
$\Phi_Q \setminus \Phi_L$ as follows. 
First, let $\alpha_1:=\gamma_1,\ldots, \alpha_r:=\gamma_r$. 
Then, let $\alpha_{r+1},\ldots, \alpha_n$ be the remaining 
roots in $\Phi_Q \setminus \Phi_L$, numbered in such a 
way that 
\begin{equation*}
\alpha_i < \alpha_j  \quad \mbox{for all} \,\, r+1 \leq i<j \leq n
\label{E: rootorder}
\end{equation*}
(with respect to the 
lexicographic order above). 
The vectors 
$F_1:=F_{\alpha_1},\ldots, F_n:=F_{\alpha_n}$ form a basis for $\fp^-$.

We now shift focus to the level of Lie groups.
Let $G$ be a 1-connected complex Lie group with Lie algebra
$\g$. Define the subgroups $B:=N_G(\mathfrak{b})$, $Q:=N_G(\fq)$, 
$L:=N_G(\fp^-) \cap N_G(\fl) \cap N_G(\fp^+)$ with 
Lie algebras $\mathfrak{b}, \fq$, and $\fl$, respectively.
Let $P^-:=\langle \exp \fp^- \rangle$ be the integral subgroup 
of $G$ generated by $\fp^-$.
Moreover, we put $B_L:=N_L(\fh \oplus \fn_L^+) \subseteq L$. 
Then $B_L$ is a Borel subgroup of $L$. Also, let
$N_L^+$ and $N_L^-$ be the integral subgroups of 
$L$ with Lie algebras $\fn_L^+$ and $\fn_L^-$, 
respectively.

We will be concerned with the homogeneous space $X:=G/Q$.
From the point of view of symplectic geometry it is 
convenient to describe $X$ as a homogeneous space under a 
compact Lie group. For this purpose, let
$U:=\exp \langle \fu \rangle \subseteq G$ be the integral
subgroup of $G$ generated by $\fu$.
Then, since $G/Q$ is connected, we have $$X=G/Q=U/K,$$ where 
$$K:=Q \cap U.$$ 
Notice that $K$ has Lie algebra $\fk$.

We will now take a closer look at a particular 
choice of local coordinates for $X$. The multiplication map
\begin{equation*}
P^- \times Q \rightarrow G, \quad (p,q) \mapsto pq
\end{equation*}
is holomorphic and injective with open image.
Moreover the exponential map
\begin{equation*}
\exp: \fp^- \rightarrow P^-
\end{equation*}
is a biholomorphic isomorphism.
It follows that the map
\begin{equation*}
\fp^- \rightarrow P^-Q/Q \subseteq X, \quad 
z_1F_1+\cdots+z_nF_n  \mapsto \exp(z_1F_1+\cdots +z_nF_n)Q
\label{E: eta}\\
\end{equation*}
is an injective holomorphic map with open image.

\section{Preliminaries from Jordan theory}
Recall that there is a one-to-one correspondence between Hermitian symmetric spaces of compact type and semisimple complex Jordan pairs with positive Hermitian involution. In the following we indicate how to obtain a Jordan pair from the Lie algebra $\mf g$. The converse direction is given by the so called Kantor-Koecher-Tits construction, for which we refer to \cite{Be00}. Our main reference is \cite{Lo75, Lo77}, and in particular we use the list of Jordan identities in \cite{Lo77} and refer to single identities by \JP{xy}.

Consider the decomposition $\mf g = \mf p^+\oplus\mf l\oplus\mf p^-$. Then, the Lie bracket defines quadratic operators
\[
	Q^\mp:\mf p^\pm\to\Hom(\mf p^\mp,\mf p^\pm),\
	x\mapsto Q^\pm_x
	\quad\text{with}\quad
	Q^\mp_x(y):=-\tfrac{1}{2}[[x,y],x]\;.
\]
This defines a Jordan pair structure on $(\mf p^-,\mf p^+)$. For convenience, we set $(V,V'):=(\mf p^-,\mf p^+)$, omit the indices $\pm$ on the quadratic operators and define operators via the relations
\[
	\JTP{x}{y}{z} := D_{x,y}z := Q_{x,z}y := Q_{x+z}y-Q_xy-Q_zy = -[[x,y],z]\;.
\]
The context determines the domains of these operators, \emph{e.g.}\ $\JTP{\;}{}{}$ is a trilinear map from $V\times V'\times V$ to $V$ (resp. from $V'\times V\times V'$ to $V'$). We also need the \emph{Bergman operator} $\B{x}{y}$ which is defined for all pairs $(x,y)\in V\times V'$ by
\begin{align*}
	\B{x}{y} &:= \Id - D_{x,y} + Q_xQ_y\in\End(V).
\end{align*}
The pair $(x,y)$ is called \emph{quasi-invertible} if $\B{x}{y}$ is invertible, and then
\begin{align}\label{eq:QuasiInverse}
	x^y:=\B{x}{y}^{-1}(x-Q_xy)\in V
\end{align}
is called the \emph{quasi-inverse} of $(x,y)$. In the same way one defines the Bergman operator $\B{y}{x}\in\End(V')$ and quasi-inverses $y^x\in V'$ for pairs $(y,x)\in V'\times V$. We note that $\B{x}{y}$ is invertible if and only if $\B{y}{x}$ is invertible.

The restriction of the Killing form $\kappa:\mf g\times\mf g\to\CC$ to the product $\mf p^-\times\mf p^+$ yields a non-degenerate pairing of $V$ and $V'$, which is given (up to a constant factor) in Jordan theoretic terms by the \emph{trace form},
\[
	\tau:V\times V'\to\CC,\; (x,y)\mapsto\Tr D_{x,y}\;,
\]
where $\Tr$ denotes the usual trace of linear operators on $V$. This turns $(V, V')$ into a semisimple complex Jordan pair.

The Cartan involution $\theta$ of $\mf g$ restricted to $\mf p^\mp$ yields antilinear isomorphisms $V\rightleftarrows V'$, which are both denoted by $\overline x:=\theta(x)$ and that satisfy $Q_{\overline x}\overline y=\overline{Q_xy}$. Moreover, the map
\begin{align}\label{eq:innerproduct}
	(\,|\,):V\times V\to\CC,\ (x,z)\mapsto (x|z):=\tau(x,\overline z)
\end{align}
is a positive definite inner product on $V$, and hence $x\mapsto\overline x$ is a \emph{positive Hermitian involution} on $(V,V')$. By means of this involution, we may identify $V$ with $V'$.

\subsection{Vector fields and group actions}\label{sec:VectorFieldsAndGroupActions}
As in \eqref{E: eta}, we may identify $V=\mf p^-$ via the exponential map with an open and dense subset of the compact Hermitian symmetric space $X=G/Q$, i.e., $V\hookrightarrow X$ by $x\mapsto \exp(x)Q$. In this way, the automorphism group $G$ acts on $V$ by birational maps, and elements of its Lie algebra $\mf g$ can be identified with vector fields on $V$, which turn out to be at most quadratic polynomials. Indeed, according to the decomposition $\mf g = \mf p^+\oplus\mf l\oplus\mf p^-$ we have isomorphisms
\begin{align}\label{eq:LieAlgebraVsVectorField}
	\begin{aligned}
	\mf p^+&\cong\Set{q_v(x):=Q_xv}{v\in V'}\;,\\
	\mf p^-&\cong\Set{u(x):=u}{u\in V}\;,
	\end{aligned}
\end{align}
and the vector fields corresponding to $\mf l$ are the \emph{derivations} on $V$, i.e., linear maps $T\in\End(V)$, satisfying
\[
	T\JTP{x}{y}{z} = \JTP{Tx}{y}{z} - \JTP{x}{T^\#y}{z} + \JTP{x}{y}{Tz}\;,
\]
for all $x,z\in V$, $y\in V'$, where $T^\#\in\End(V')$ is the adjoint map of $T$ with respect to the trace from $\tau$. As an example, for any pair $(x,y)\in V\times V'$, the operator $D_{x,y}$ is a derivation with $(D_{x,y})^\# = D_{y,x}$. Within the context of Jordan theoretic arguments, we identify the Lie algebra $\mf g$ with its realization as vector fields on $V$, so an element $X\in\mf g$ is a vector field $\zeta_X:V\to V$ of the form $\zeta_X(x)=u+Tx+q_v(x)$ with $u\in V$, $T\in\mf l$ and $v\in V'$. In order to obtain Lie algebra isomorphisms, we note that the commutator of vector fields $\zeta,\eta\in\mf g$ is given by
\begin{align}\label{eq:VecFieldComm}
	[\zeta,\eta](x) = d\zeta(z)\cdot\eta(z) - d\eta(z)\cdot\zeta(z)\;,
\end{align}
which differs by sign from the usual convention for the Lie bracket of vector fields. In detail, the commutator of two elements $X_1= u_1+T_1+q_{v_1}$ and $X_2 = u_2+T_2+q_{v_2}$ is given by
\[
	[X_1,X_2] = (T_1u_2-T_2u_1) + \big(D_{u_2,v_1} + [T_1,T_2] - D_{u_1,v_2}\big)
							+(q_{T_2^\#v_1} - q_{T_1^\#v_2})\;,
\]
and the Killing form $\kappa$ on $\mf g$ translates to 
\begin{align}\label{eq:KillingForm}
	\kappa(X_1,X_2) = \kappa_\mf l(T_1,T_2) + 2\,\Tr(T_1T_2) - 2\,\tau(u_1,v_2) - 2\,\tau(u_2,v_1)\;,
\end{align}
where $\kappa_\mf l$ denotes the Killing form on $\mf l$. For both formulas, see \emph{e.g.}\ \cite[\S7]{Sa80}.

The birational group action of $\exp(\mf p^\pm)$ on $V\subseteq X$ is given and denoted by
\begin{align*}
	&\trans{u}(x):=\exp(u)x = x+u & &\text{for\quad$u\in V$ (\emph{translation})},\\
	&\qtrans{v}(x):=\exp(q_v)x = x^v & &\text{for\quad$v\in V'$ 
		(\emph{quasi-translation})}.
\end{align*}
The subgroup $L\subseteq G$ is identified with the identity component of the automorphism group $\Aut(V,V')$ of the Jordan pair, which consists of linear automorphisms $h\in\GL(V)$ satisfying $h\JTP{x}{y}{z} = \JTP{hx}{h^{-\#}y}{hz}$ for all $x,z\in V$, $y\in V'$, where $h^{-\#} := (h^\#)^{-1}$ and $h^\#$ is the adjoint map of $h$ with respect to the trace form $\tau$. Therefore, $h\in L$ acts on $V\subseteq X$ by linear transformations $x\mapsto hx$. As an example, for quasi-invertible pairs $(x,y)\in V\times V'$, the Bergman operator $\B{x}{y}$ is a Jordan pair automorphism with $(\B{x}{y})^{\#} = \B{y}{x}$.

The Cartan involution $\theta$ corresponding to the compact real froms $\mf k\subseteq\mf l$ and $\mf u\subseteq\mf g$ translates to $\theta(u+T+q_v) = \overline v - T^* + q_{\overline u}$, where $T^*$ denotes the adjoint of $T$ with respect to the inner product \eqref{eq:innerproduct} on $V$. Therefore, 
\[
	\mf k = \Set{T\in\mf l}{T = -T^*}\;,\quad
	\mf u = \Set{u+T+q_{\overline u}}{u\in V,\ T\in\mf k}\;.
\]
In addition, we note that the centre of $\mf k$ is given by $\mf z(\mf k) = \RR(i\Id)$,
and the Lie group $K$ is the connected component of the group of automorphisms $h\in\Aut(V,V')$ satisfying $h = h^{-*}$, i.e., unitary automorphisms.

\subsection{Idempotents, Peirce decomposition, and rank}
An \emph{idempotent} is a pair $\b e = (e,e')\in V\times V'$ satisfying the relations $Q_ee' = e$ and $Q_{e'}e = e'$. Let $\Idem\subset V\times V'$ denote the set of idempotents. For $\b e\in\Idem$, the operators $D_{e,e'}\in\End(V)$ and $D_{e',e}\in\End(V')$ are diagonalizable with spectra in $\{0,1,2\}$, and the decomposition into eigenspaces,  
\begin{align*}
	V = V_2(\b e)\oplus V_1(\b e)\oplus V_0(\b e)\;,\quad
	V' = V'_2(\b e)\oplus V'_1(\b e)\oplus V'_0(\b e),
\end{align*}
is called the \emph{Peirce decomposition} with respect to $\b e$. We note that in general, $V'_k(\b e)$ differs from the image of $V_k(\b e)$ under the involution of $(V,V')$, i.e., $V'_k(\b e)\neq\overline{V_k(\b e)}$. The \emph{Peirce spaces} $V_k:=V_k(\b e)$, $V'_k:=V'_k(\b e)$ are subject to the following multiplication rules (the \emph{Peirce rules})
\begin{align*}
	\JTP{V_i}{V'_j}{V_k}\subseteq V_{i-j+k},\
	\JTP{V_2}{V'_0}{V} = \JTP{V_0}{V'_2}{V} = \{0\}\;,
\end{align*}
where $V_\ell = \{0\}$ and $V'_\ell=\{0\}$ if $\ell\notin\{0,1,2\}$. In particular, $(V_k,V'_k)$ is a subpair of $(V,V')$. Two idempotents $\b e=(e,e'),\ \b c=(c,c')$ are (\emph{strongly}) \emph{orthogonal} if $c\in V_0(\b e)$ or equivalently $e\in V_0(\b c)$. In this case, the sum $\b e +\b c = (e+c,e'+c')$ is also an idempotent. A non-zero idempotent is called \emph{primitive}, if it is not the sum of two orthogonal non-zero idempotents. A \emph{frame of idempotents} $(\b e_1,\ldots,\b e_r)$ is a maximal system of primitive orthogonal idempotents. The length $r$ of a frame of idempotents is an invariant of the Jordan pair $(V,V')$, called the \emph{rank} and denoted by $\rank V:=r$.

For a system of (pairwise) orthogonal idempotents $(\b e_1,\ldots,\b e_k)$, so in particular for a frame, the operators $(D_{e_\ell,e_\ell'})_{\ell=1,\ldots,k}$  form a commuting set of diagonalizable operators, and hence induce the \emph{joint Peirce decomposition}
\begin{align*}
	V = \bigoplus_{0\leq i\leq j\leq k} V_{ij}
	\quad\text{with}\quad
	V_{ij} = \Set{x\in V}{\JTP{e_\ell}{e'_\ell}{x}
														= (\delta_{i\ell}+\delta_{j\ell})\,x\text{ for all $\ell$}}\;,
\end{align*}
and likewise for $(D_{e_\ell',e_\ell})_{\ell=1,\ldots,k}$ and $V'$. Setting $V_{ji}:=V_{ij}$ and $V'_{ji}:=V'_{ij}$ for $i\neq j$, the Peirce rules refine to the \emph{joint Peirce rules}
\[
	\JTP{V_{ij}}{V'_{jk}}{V_{k\ell}}\subseteq V_{i\ell}\;,
\]
and all other types of products vanish. Again, we point out that the image of $V_{ij}$ under the involution of $(V,V')$ in general differs from $V'_{ij}$, unless we consider a special class of idempotents, namely those defined by tripotents, which we discuss in the next section.

If $(V,V')$ is simple of rank $r$, then, for any primitive idempotent $\b e\in\Idem$, we set
\begin{align}\label{eq:StructureConst}
	p:=\tau(e,e') = 2+\dim V_1(\b e)
\end{align}
which is independent of the choice of $\b e$. This structure constant of $(V,V')$ appears in subsequent formulas which involve the Killing form of $\mf g$.

We also need the notion of \emph{rank} for arbitrary elements $x\in V$ or $y\in V'$. For $x\in V$ the subspace $[x]:=Q_xV'\subseteq V$ is called the \emph{principal inner ideal} generated by $x$. The \emph{rank} of $x$, $\rank x$, is defined as the maximum length of all chains $[x_0]\subsetneq[x_1]\subsetneq\cdots\subsetneq[x_k]$ with $x_i\in[x]$. Similarly, one defines $[y]:=Q_yV\subseteq V'$ and $\rank y$ for $y\in V'$. If $\b e=(e,e')$ is an idempotent, then the Peirce rules imply $[e] = V_2(\b e)$ and $[e'] = V'_2(\b e)$, and it turns out \cite[\S3]{Lo91}, that $\rank e = \rank e'$. We therefore define $\rank\b e:=\rank e$ and call this the \emph{rank of the idempotent $\b e$}. The set $\Idem$ of idempotents therefore decomposes into subset of constant rank idempotents, denoted by $\Idem_k:=\Set{\b e\in\Idem}{\rank\b e = k}$. Since we assume $(V,V')$ to be finite dimensional and simple, each element $e\in V$ admits a \emph{completion to an idempotent}, i.e., an element $e'\in V'$ such that $(e,e')\in\Idem$. From this, it follows that the decomposition of an idempotent $\b e$ of rank $k$ into primitive orthogonal idempotents has exactly $k$ summands. Therefore, the maximum of all ranks of elements in $V$ coincides with the rank of $V$ as it is defined above.

\subsection{Tripotents and spectral decomposition}
The involution on $(V,V')$ admits the definition of (odd) powers of elements, namely for $x\in V$ define $x^{(1)}:= x$ and inductively $x^{(2k+1)}:=Q_x\overline{x^{(2k-1)}}$ for $k\geq 1$. An element $e\in V$ is called a \emph{tripotent}, if $e^{(3)}= e$, i.e., $e = Q_e\overline e$. Equivalently, $e$ is a tripotent if and only if $(e,\overline e)$ is an idempotent. In particular, all notions defined in the last section apply to the idempotent $(e,\overline e)$. Without causing ambiguities in notation, we may identify $e$ with $(e,\overline e)$ if necessary. Concerning the Peirce decomposition, we note that $V'_k(e)=\overline{V_k(e)}$ for a tripotent $e$, and $V'_{ij} = \overline{V_{ij}}$ for a system of orthogonal tripotents $(e_1,\ldots,e_k)$. The set of tripotents is denoted by $\Tri\subseteq V$.

For the explicit description of the moment map on $X$, we will make use of the following spectral theorem \cite[\S3.12]{Lo77}.
\begin{thm}[Spectral decomposition]
	Let $(V,V')$ be a finite dimensional semisimple Jordan pair with positive Hermitian 
	involution. Then every element $x\in V$ admits a unique decomposition
	\[
		x = \sigma_1e_1+\cdots+\sigma_k e_k\;,\quad\sigma_1>\cdots>\sigma_k>0\;,
	\]
	where the $e_i$ are pairwise orthogonal non-zero tripotents which are real linear combinations
	of powers of $x$, and $\sigma_i\in\RR$.
\end{thm}

\subsection{Idempotents and roots}\label{subsec:Correspondence}
Idempotents are related to $\mf{sl}_2$-triples in $\mf g$ in the following way: If $(e,e')\in\Idem $ is an idempotent, then $(E,H,F):=(q_{e'},D_{e,e'},-e)$ is an $\mf{sl}_2$-triple in $\mf g$ with $q_{e'}\in\mf p^+$, $D_{e,e'}\in\mf l$ and $-e\in\mf p^-$ (constant vector field). Indeed, according to \eqref{eq:VecFieldComm} and using \JP{12}, it follows that
\begin{align*}
	&[D_{e,e'},q_{e'}] = q_{D_{e',e}e'} = 2\,q_{e'}\,,\quad\\
	&[D_{e,e'},-e] = -D_{e,e'}e = -2\,e\,,\quad\\
	&[-e,q_{e'}] = D_{e,e'}\;.
\end{align*}
Conversely, if $(E,H,F)$ is an $\mf{sl}_2$-triple in $\mf g$ with $E\in\mf p^+$, $H\in\mf l$ and $F\in\mf p^-$, then the identifications $V'\cong\mf p^+$ and $V\cong\mf p^-$ in \eqref{eq:LieAlgebraVsVectorField} yield an idempotent $(e,e')\in\Idem$ with $E=q_{e'}$, $F = -e$  and $H = D_{e,e'}$.

This correspondence between idempotents and certain $\mf{sl}_2$-triples in $\mf g$ also applies to tripotents: Starting	with a tripotent $e\in\Tri$, we obtain the $\mf{sl}_2$-triple $(E,H,F):=(q_{\overline e},D_{e,\overline e},-e)$ with the additional property $\theta(E) = -F$, where $\theta$ denotes the Cartan involution on $\mf g$, and hence	$iH\in\mf k$. Conversely, any $\mf{sl}_2$-triple $(E,H,F)$ with $E\in\mf p^+$ and $\theta(E) = -F$ yields a tripotent $e\in\Tri$ corresponding to $E = q_{\overline e}$.

In particular, the $\mf{sl}_2$-triples associated to the system of strongly orthogonal roots $\gamma_1,\ldots,\gamma_r$ yield a system $(e_1,\ldots,e_r)$ of tripotents, and it is straightforward to see that strong orthogonality of the roots is equivalent to strong orthogonality of the tripotents. Therefore,	$(e_1,\ldots,e_r)$ is a frame associated to the system of strongly orthogonal roots. We summarize the situation by
\begin{align*}
	(E_{\gamma_j},H_{\gamma_j},F_{\gamma_j}) = (q_{\overline e_j},D_{e_j,\overline e_j},-e_j)\;.
\end{align*}
Moreover, using the Killing form \eqref{eq:KillingForm}, the relation $\gamma_j = c\cdot\kappa(H_{\gamma_j},-)$ with $c = 2/\kappa(H_{\gamma_j},H_{\gamma_j})$ yields
\begin{align}\label{eq:GammaIdentity}
	\gamma_j(T) = \tfrac{1}{p}\,\tau(Te_j,\overline e_j)
	\quad\text{for all}\quad T\in\mf l\;,
\end{align}
where $p$ is the structure constant defined by \eqref{eq:StructureConst}. This is the Jordan theoretic description of the strongly orthogonal roots.

Recall the decomposition \eqref{eq:MaxAbDecomp} of the maximal abelian subalgebra $\mf t$ of $\mf k$ into $\mf t = \mf s\oplus\mf s^\bot$. In Jordan theoretic terms, we obtain
\begin{align*}
	\mf s = \left\langle i\,D_{e_j,\overline e_j},\; i\Id_V\,\mid\,j=1,\ldots,r\right\rangle_\RR,
\end{align*}
and the condition $\mf s^\bot\subseteq\Set{T\in\mf l}{\gamma_j(T) = 0}$ translates to
\begin{align*}
	\mf s^\bot \subseteq \Set{T\in\mf k}{Te_j=0\text{ for all j}}.
\end{align*}

\subsection{Determinants}
We have to deal with two kinds of determinants. On the one hand, there is the \emph{Jordan pair determinant} $\Delta:V\times V'\to\CC$ associated to a Jordan pair $(V,V')$, often also called the \emph{generic minimum polynomial}, cf.\ \cite[\S16]{Lo75}.
On the other hand, let $\b e=(e,e')$ be an idempotent in $(V,V')$, then the Peirce $2$-space $V_2(e)$ becomes a unital Jordan algebra with product $x\circ z :=\tfrac{1}{2}\JTP{x}{e'}{z}$ and unit element $e$. The Jordan algebra determinant corresponding to $V_2(e)$ is denoted by $\Delta_{\b e}:V_2(\b e)\to\CC$. Likewise, $V'_2(\b e)$ is a Jordan algebra with product $y\circ w:=\tfrac{1}{2}\JTP{y}{e}{w}$, unit element $e'$, and Jordan algebra determinant $\Delta'_{\b e}:V'_2(\b e)\to\CC$. The connection between Jordan pair determinant and Jordan algebra determinants is given as follows:

\begin{lemma}\label{lem:JordanDeterminants}
	If $\b e = (e,e')$ is an idempotent, then
	\begin{align}\label{eq:JordanDeterminants}
		\Delta_\b e(x) = \Delta(e-x,e')\;,\quad
		\Delta'_\b e(y) = \Delta(e,e'-y)
	\end{align}
	for all $x\in V_2(\b e)$ and $y\in V'_2(\b e)$.
\end{lemma}
\begin{proof}
Each identity is a consequence of the relation between Jordan algebra inverses and quasi-inverses, \emph{e.g.}\ $x^{-1} = e + (e-x)^{e'}$ for invertible $x\in V_2(\b e)$. We omit the details.
\end{proof}

In the following, we use \eqref{eq:JordanDeterminants} to extend the Jordan algebra determinants $\Delta_{\b e}$ and $\Delta'_{\b e}$ to polynomial maps on $V$ and $V'$. By abuse of notation, these extensions are also denoted by $\Delta_{\b e}$ and $\Delta'_{\b e}$. We note that if $V = V_2(\b e)\oplus V_1(\b e)\oplus V_0(\b e)$ is the Peirce decomposition with respect to $\b e$, then $\Delta_{\b e}$ vanishes on $V_1(\b e)\oplus V_0(\b e)$. Analog results hold for $\Delta'_{\b e}$. For later use, we note the following relation between the rank of idempotents and zeros of Jordan algebra determinants:
\begin{lemma}\label{lem:RankConditions}
	Let $x\in V$ be a fixed element, and $k\in\NN$. Then, $\Delta_\b c(x) = 0$ 
	for all idempotents $\b c$ of rank $k$ if and only if $k>\rank x$. The same holds for $y\in V'$
	and $\Delta'_\b c(y) = 0$ in place of $\Delta_\b c(x) = 0$.
\end{lemma}
\begin{proof}
First assume that $k\leq\rank x$. Let $\b e = (e,e')$ be a completion of $e=x$ to an idempotent, and let $\b e = \b e_1+\cdots+\b e_\ell$ be a decomposition into primitive orthogonal idempotents. Then, $\ell=\rank\b e=\rank x\geq k$, and hence $\b c:=\b e_1+\cdots+\b e_k$ is well-defined and satisfies $\Delta_\b c(x) = \Delta_\b c(e) = 1$. This proves the 'only if' part. For the converse direction assume $\Delta_\b c(x)\neq 0$ for some $\b c\in\Idem_k$. Let $x = x_2+x_1+x_0$ be the components of $x$ in the Peirce decomposition of $V$ with respect to $\b c$. Then, $\Delta_\b c(x) = \Delta_\b c(x_2)$, and this is non-vanishing if and only if $x_2$ is invertible in the unital Jordan algebra $V_2(\b c)$. Furthermore, this is equivalent to the identity $[x_2]=V_2(\b c)$, where $[x_2]=Q_{x_2}V'$ is the principal inner ideal corresponding to $x_2$. Since $V_2(\b c)=[c]$, this implies that $\rank x_2 = \rank c =\rank\b c$. Now the statement follows from the inequality $\rank x_2\leq\rank x$, cf.\ \cite[\S3]{Lo91}.
\end{proof}

\subsection{Jordan theoretic model of $X$}
The concept of quasi-inverses \eqref{eq:QuasiInverse} provides a Jordan theoretic model for the compact Hermitian symmetric space $X=G/Q$ (due to O.~Loos \cite{Lo77}): For each $a\in V'$ the map $\iota_a:V\to X$ given by $\iota_a(x) =\exp(q_a)\exp(x)Q = \qtrans{a}\trans{x}Q$ is an open and dense imbedding of $V$ in $X$. This yields an open covering of $X$ by the subsets $X_a:=\iota_a(V)$, $a\in V'$. It turns out that $X_a\cap X_b$ is the image of $\set{x\in V}{(x,a-b)\text{ quasi-invertible}}$ under $\iota_a$, and the transition map $\varphi^a_b=\iota_b^{-1}\circ\iota_a$ is given by
\begin{align*}
	\varphi^a_b(x) = x^{a-b}\;.
\end{align*}
This description of $X$ may be summarized in the equivalence relation $X \cong (V\times V')/R$ with
\begin{align}\label{eq:ElemEquivRel}
					(x,a)\,R\,(\tilde x,b)\iff
					\left\{
					\begin{aligned}
						&(x,a-b)\text{ is quasi-invertible}\\
						&\text{and }\tilde x = x^{a-b}\;.
					\end{aligned}\right.
\end{align}
The equivalence class of an element $(x,a)$ is denoted by $\GP{x}{a}$. For $a=0$, the imbedding $\iota_0$ is just the standard imbedding of $V\cong\mf p^-$ into $X$, which we have already used in Section~\ref{sec:VectorFieldsAndGroupActions}. In the following, we write $V\subseteq X$ for the identification of $V$ with its image under $\iota_0$ in $X$.

Some questions require yet another description of the elements of $X$. Consider the description via equivalence classes as it is given in \eqref{eq:ElemEquivRel}. Whereas the chart maps concerns elements with fixed second entry $a\in V'$, i.e., $x\mapsto\GP{x}{a}$, the following proposition selects for each element in $X$ a representative which is adapted to arguments concerning the action of the automorphism group of $(V,V')$ on $X$.

\begin{prop}\label{prop:Representatives}
	Let $\GP{x}{a}$ be an element of $X=(V\times V')/R$. Then, there exist an idempotent
	$\b e=(e,e')\in\Idem$ and an element $z\in V_0(\b e)$ such that $\GP{x}{a} = \GP{e+z}{e'}$.
	Moreover, the idempotent $\b e$ can be chosen to be tripotent, i.e., $\b e = (e,\overline e)$.
\end{prop}
\begin{proof}
The existence of a representative for $\GP{x}{a}$ of the form $(e+z,e')$ with idempotent $\b e$ and $z\in V_0(\b e)$ is proved in \cite{Lo94}, see Theorem~3.8, Proposition~4.6, and Theorem~4.7 therein. The possibility to choose $\b e$ to be tripotent, follows from the fact that in our setting ($(V,V')$ being finite dimensional and simple), for any idempotent $\b e$ there exists a tripotent $c$ such that $V_2(\b e) = V_2(c)$. Therefore, we may assume that the idempotent $e = (e_+,e_-)$ used in the last part of Proposition~6.5 in \cite{Lo94} is in fact a tripotent.
\end{proof}

\section{Line bundles and representation spaces for $G$}

\subsection{Prequantum bundles on $G/Q$}\label{subsec:PrequantumBundle}
Let $-\lambda \in \fh^*$ be the fundamental weight 
associated to the simple root $\beta_1$. 
Let $\chi_\lambda: Q \rightarrow \C^\times$ be the 
holomorphic character determined by the condition 
$d\chi_\lambda(e)\mid_{\fh}=\lambda$, and let 
$\mathscr{L}:=G \times_Q \C \rightarrow X$ be 
the line bundle associated to $\chi_\lambda$. 
We shall be concerned with the family $H^0(X, \mathscr{L}^k)$, 
for $k \in \N$, of $G$-representations, and with the
decomposition under the group $L$. 
Notice that $K$ is a maximal compact subgroup 
of $L$, so that the decompositions under $L$ amount 
to the decompositions under $K$. 

In order to realize the homogeneous space $X$ as 
a coadjoint $U$-orbit in $\fu^*$ we extend 
the $\C$-linear functional $\lambda \in \fh^*$ to a 
functional on $\g$ by requiring that it annihilate 
all root spaces $\g_\alpha$. 
This convention will hereafter be used for
extending linear functionals on $\fh$ to linear functionals 
on $\g$. Similarly we extend $\R$-linear functionals
on $\ft$ to $\R$-linear functionals on $\fu$. 

Notice that $\lambda$ has imaginary values on 
$\ft$, so that $i\lambda$ restricts to a real-valued 
$\R$-linear functional on $\ft$. We will write 
$i\lambda$ instead of $i\lambda \mid_{\fu}$ for 
the induced $\R$-linear functional on $\fu$.
Using the above conventions, $X$ can be realized as 
the coadjoint orbit, $\mathcal{O}_\lambda \subseteq \fu^*$, 
of $i\lambda$. When $\mathcal{O}_\lambda$ is equipped with the 
Kostant-Kirillov symplectic form, $\omega_\lambda$,  
the action of $U$ is Hamiltonian with moment map 
\begin{equation*}
\mu: \mathcal{O}_\lambda \rightarrow \fu^*
\end{equation*}
being the inclusion $\mathcal{O}_\lambda \subseteq \fu^*$. 
The action of the subgroup $K$ is also Hamiltonian, with 
moment map
\begin{equation*}
\mu_\fk: \mathcal{O}_\lambda \rightarrow \fk^*, \quad \mu_\fk(f)(x):=f(x), 
\quad f \in \mathcal{O}_\lambda, \quad  x \in \fk.
\end{equation*}
The line bundle $\mathscr{L}$ is a prequantum line bundle 
for the Kostant-Kirillov form on $\mathcal{O}_\lambda$. 
In order to interpret the tensor powers $\mathscr{L}^k$ 
as prequantum line bundles on $X \cong \mathcal{O}_\lambda$, 
we use the natural isomorphism 
$\mathcal{O}_\lambda \cong \mathcal{O}_{k\lambda}, f \mapsto kf$, 
to identify the symplectic manifolds  
$(\mathcal{O}_{k\lambda}, \omega_{k\lambda})$ and
$(\mathcal{O}_\lambda, k\omega_\lambda)$. 
Then $\mathscr{L}^k$ is a prequantum line bundle for 
$(\mathcal{O}_\lambda, k\omega_\lambda)$. 
The moment map for the $K$-action on 
$(\mathcal{O}_\lambda, k\omega_\lambda)$ is 
$$\mu_\fk^k:=k\mu_\fk.$$

\subsection{Trivialization of sections}
Consider the dual space $H^0(X, \mathscr{L})^*$, which
is a highest weight module of highest weight $-\lambda$. 
We make a specific choice of 
a highest weight vector. For this purpose we identify 
the space of holomorphic sections $H^0(X, \mathscr{L})$
with the space of $Q$-equivariant holomorphic functions 
$F: G \rightarrow \C$ 
having the $Q$-equivariance property
\begin{equation*}
F(gq)=\chi_\lambda^{-1}(q)F(g), \quad g \in G, \quad q \in Q.
\end{equation*}
Define $w_1 \in H^0(X, \mathscr{L})^*$ as the linear 
functional $F \mapsto \mbox{ev}_e(F):=F(e)$.
Then $w_1$ is a cyclic vector for 
$H^0(X, \mathscr{L})^*$ as a $U(\g)$-module, as well 
as a $U(\mathfrak{p}^-)$-module.
For $k \in \N$, the vector $w_k:=w_1^{\otimes k} \in 
H^0(X, \mathcal{L}^k)^*$ is then a highest weight 
vector for $H^0(X, \mathscr{L}^k)^*$.
Given these normalizations of highest weight vectors of 
the spaces  $H^0(X, \mathscr{L}^k)^*$ we now
consider local trivializations of the bundles 
$\mathscr{L}^k$.

The principal bundle $q: G \rightarrow G/Q$ is trivial 
over the open set $P^-Q/Q$. Hence
the bundles $\mathscr{L}^k$, being associated to this  
principal bundle, are also trivial over $P^-Q/Q$. 
If we identify a section $\varphi$ of $\mathscr{L}^k$ with 
a linear functional on the dual space $H^0(X, \mathscr{L}^k)^*$, 
then the realization of $\varphi$ as a $Q$-equivariant holomorphic 
function, $F$, on $G$ is given by
$$F(g)=\varphi(g.w_k).$$
The trivialization of a section, viewed as a $Q$-equivariant
holomorphic function $F: G \rightarrow \C$, is given by
the restriction of $F$ to $P^-$. 
In the local coordinates \eqref{E: eta}, the restriction of 
$F$ to $P^-$ is given by 
the function 
\begin{equation}
f(z_1,\ldots, z_n)=\varphi(\exp (z_1F_1+\cdots +z_nF_n).w_k), 
\quad z \in \fp^-.
\label{E: triv}
\end{equation}

\begin{prop} \label{P: monomweight}
Let $s \in H^0(X,\mathscr{L}^k)$ be a 
weight vector of weight $\delta$.
If scalar multiples of the two monomial terms 
$z^a$ and $z^b$ occur in the polynomial that trivializes
$s$, then the identity
\begin{equation*}
k\lambda+\sum_{j=1}^n a_j\alpha_j
=k\lambda+\sum_{j=1}^n b_j\alpha_j.
\end{equation*}
holds in the weight lattice.
Moreover, $\delta=k\lambda+\sum_{j=1}^n a_j\alpha_j$.
\end{prop}

\begin{proof}
The holomorphic section $s$ corresponds to a
linear functional $\varphi$ on the space 
$H^0(X,\mathscr{L}^k)^*$. The trivialization of $s$  
is then given by the polynomial function
\begin{align*}
f(z_1,\ldots, z_n)&=\varphi(\exp (z_1F_1+\cdots +z_nF_n).w_k)\\
&=\sum_{c \in \N_0^n} \frac{1}{c_1!\cdots c_n!} 
z^c\varphi(F^c.w_k),
\end{align*}
where $F^c:=F_1^{c_1}\cdots F_n^{c_n}$.
Since each vector $F^c.w_k$ is a weight vector (of
weight $-k\lambda-\sum_{j=1}^nc_j\alpha_j$), the functional $\varphi$ 
can only have nonzero values on 
vectors $F^c.w_k$ of the fixed weight $-\delta$. 
\end{proof}

We now define a \emph{valuation-like function} 
\begin{equation*}
v: \bigsqcup_{k \in \N} H^0(X, \mathscr{L}^k) \setminus \{0\}
\rightarrow \N_0^n,
\end{equation*}
i.e., a function satisfying the properties
\begin{align*}
&v(st)=v(s)+v(t), \quad s 
\in H^0(X, \mathscr{L}^k) \setminus \{0\},
\,t\in H^0(X, \mathscr{L}^\ell) \setminus \{0\},\\
&v(s+t) \geq \mbox{min}(v(s),v(t)), \quad s,t \in 
H^0(X, \mathscr{L}^k) \setminus \{0\}, \,\,s+t \neq 0,\\
&v(\lambda s)=v(s), \quad s \in H^0(X, \mathscr{L}^k) 
\setminus \{0\}, \lambda \in \C^\times.
\end{align*}
If $s \in H^0(X, \mathscr{L}^k) \setminus \{0\}$ 
is trivialized as the polynomial $f=\sum_{a \in \N_0^n}c_az^a$, 
let
\begin{equation} \label{E: valdef}
v(s):=\text{min}\{a \in \N_0^n \mid c_a \neq 0\},
\end{equation}
where the minimum refers to the inverse lexicographic order 
on $\Z^n$.
\begin{rem}
The function $v$, although it seems to depend 
on the particular local coordinates chosen, has a global
geometric meaning. In fact, $v$ can be interpreted as defining the 
``successive orders of vanishing'' of $s$ 
along a flag of irreducible subvarieties
$X_0 \subseteq \ldots \subseteq X_{n-1} \subseteq X_n=X$, 
where $\mbox{dim} X_i=i$, in the setting defined by 
Okounkov (\cite{Ok}), and later developed by 
Lazarsfeld-Musta{\c{t}}{\u{a}} (\cite{LM09}). This interpretation 
will however not play any role in the rest of this paper.
Instead, the particular local expression \eqref{E: valdef} 
will be useful.
\end{rem}
To the function $v$ we attach the semigroup
\begin{equation}
S(\mathscr{L}, N^+_L, v):=
\{(k,v(s)) \mid s \in H^0(X,\mathscr{L}^k)^{N^+_L} \setminus \{0\}\} 
\subseteq \N \times \N_0^{n}, \label{E: semigroup}
\end{equation}
and the closed convex cone, $C(\mathscr{L}, N^+_L, v)$, in 
$\R \times \R^n$ which is generated by 
the semigroup $S(\mathscr{L}, N^+_L, v)$.  
Finally, we define the \emph{Okounkov body}, 
\begin{equation}
\Delta(\mathscr{L}, N^+_L,v):=C(\mathscr{L}, N^+_L, v) 
\cap (\{1\} \times \R^n). \label{E: okounkov}
\end{equation}

\section{$L$-types and $L$-orbits in $X$}
Recall that $L$ is the identity component of the automorphism group of the Jordan pair $(V,V')$.
We give a geometric proof of the decomposition of $H^0(X,\linebundle)$ into irreducible $L$-modules. More precisely, we claim that the decomposition is obtained by restriction to the closed $L$-orbits in $X=G/Q$. 

\subsection{Closed $L$-orbits}
Recall that the set $\Idem\subseteq V\times V'$ of idempotents decomposes into the subsets of constant rank idempotents,
\[
	\Idem = \bigcupdot_{k=0}^r\Idem_k\quad\text{with}\quad
	\Idem_k = \Set{\b e\in\Idem}{\rank\b e = k}\;,
\]
where $r$ denotes the rank of the Jordan pair $(V,V')$. Assuming that $(V,V')$ is simple (or equivalently, that the Hermitian symmetric space $X$ is irreducible), it is well-known that the $\Idem_k$ are the connected components of $\Idem$ (with respect to the induced topology from $V\times V'$), and that the $L$-action on $V\times V'$ restricts to a transitive $L$-action on each component $\Idem_k\subseteq V\times V'$, cf.\ \cite[\S 17.1]{Lo75}. Let $N_k$ be the dimension of the Peirce $2$-space $\dim V_2(\b e)$ for some idempotent $\b e\in\Idem_k$, which is independent of the choice of $\b e$, and let $\Gr_{N_k}(V)$ be the (classical) Grassmannian manifold of $N_k$-dimensional subspaces in $V$, equipped with the natural $L$-action induced by the $L$-action on $V$. Consider the subset
\[
	\Peirce_k
		:=\Set{U\subseteq V}{U=V_2(\b e)\text{ for some $\b e\in\Idem_k$}}\subseteq\Gr_{N_k}(V)\;.
\]
Since $V_2(h\b e) = hV_2(\b e)$ for any $h\in L$ and $\b e\in\Idem_k$, the map
\[
	\pi_k:\Idem_k\to\Peirce_k,\ \b e\mapsto V_2(\b e)
\]
is $L$-equivariant, and hence $\Peirce_k$ is an $L$-orbit in $\Gr_{N_k}(V)$. In particular, it is a smooth projective variety, called the \emph{Peirce variety} of rank $k$. We call two idempotents $\b e,\b c\in\Idem_k$ \emph{Peirce equivalent}, if $\pi_k(\b e) = \pi_k(\b c)$, i.e., if their Peirce $2$-spaces in $V$ coincide.

We show that the Peirce varieties $\Peirce_k$ can be imbedded $L$-equivariantly into the compact Hermitian symmetric space $X=G/Q$. Recall the Jordan-theoretic description of $X$ via \emph{projective equivalence}, $X\cong (V\times V')/R$. The following is a crucial observation due to O.~Loos, cf.\ \cite[\S2.6]{Lo94}. 

\begin{lemma}\label{lem:EqualEquivalence}
	Two idempotents $\b e,\;\b c\in\Idem_k\subseteq V\times V'$ are Peirce equivalent	if and 
	only if they are projectively equivalent, i.e., $V_2(\b e) = V_2(\b c)$ if and only if 
	$\GP{e}{e'} = \GP{c}{c'}$.
\end{lemma}

Applying this lemma, it follows that the map
\[
	\iota_k:\Peirce_k\to X,\ U\mapsto \GP{e}{e'}\quad\text{for}\quad
	U=V_2(\b e)\text{ with $\b e = (e,e')$}
\]
is well-defined and one-to-one. We may illustrate the situation in the following commutative diagram
\begin{equation}\label{diag:PeirceImbed}
	\begin{tikzpicture}	[baseline=(current bounding box.center),>=angle 90]
		\matrix [column sep = 1.5cm, row sep = 1cm, text height=1.5ex, text depth=0.25ex]
		{
			\node [anchor=base] (I) {$\Idem_k$};   & \node [anchor=base] (V) {$V\times V'$}; \\
			\node [anchor=base] (P) {$\Peirce_k$}; & \node [anchor=base] (X) {$X$}; \\
		};
		\draw [->>] (I) to node [auto,swap] {$\pi_k$} (P);
		\draw [right hook->] (I) to node [auto] {id} (V);
		\draw [right hook->] (P) to node [auto] {$\iota_k$} (X);
		\draw [->>] (V) to node [auto] {$\pi$} (X);
		\node at (2,-0.9) {.};
	\end{tikzpicture}
\end{equation}
Moreover, since $h\b e = (he,h^{-\#}e')$ and $h\GP{e}{e'} = \GP{he}{h^{-\#}e'}$, this diagram is also $L$-equivariant, and in particular $\iota_k$ is an $L$-equivariant isomorphism onto its image
\[
	X_k:=\iota_k(\Peirce_k)\subseteq X\;,
\]
which is a closed $L$-orbit in $X$. In this way, we have identified $r+1$ different $L$-orbits in $X$. Recalling a classical result, that there are precisely $r+1$ closed $L$-orbits in $X$, we summarize the arguments of this section.

\begin{prop}\label{prop:LOrbits}
	The closed $L$-orbits in $X$ are precisely the complex analytic submanifolds
	\[
		X_k = \Set{\GP{e}{e'}}{(e,e')\in\Idem_k}\subseteq X
	\]
	for $k = 0,\ldots,r$. The $L$-orbit $X_k$ is $L$-equivariantly isomorphic via $\iota_k$ to the 
	Peirce variety $\Peirce_k$.
\end{prop}

\subsection{Restriction isomorphism}
In the following, we use the notation of diagram \eqref{diag:PeirceImbed} and consider the restrictions of the fundamental line bundle $\linebundle$ on $X$ to the closed $L$-orbits $(X_k)_{k=0,\ldots,r}$ that are described in Proposition~\ref{prop:LOrbits}. Due to \cite[Proposition~3.1]{Lo78}, the projection $\pi$ trivializes the line bundle $\linebundle$, i.e., $\pi^*\linebundle$ is trivial, and sections in $\linebundle$ can be identified with holomorphic maps on $V\times V'$ satisfying a certain cocycle condition, i.e.,
\[
	H^0(X,\linebundle)\cong \Set{f:V\times V'\to\CC\text{ hol.}}
	{\begin{aligned}
		&f(x,y) = \Delta(x,y-y')\cdot f(x^{y-y'},y')\\
		&\text{for all quasi-invertible $(x,y-y')$}
	\end{aligned}}.
\]
Let $\linebundle|_{X_k}$ denote the restriction of the line bundle to the closed $L$-orbit $X_k\subseteq X$. By compactness of $X_k$, the space $H^0(X_k,\linebundle|_{X_k})$ is finite dimensional. Moreover, $L$ acts irreducibly. Due to Proposition~\ref{prop:LOrbits} and the commutativity of \eqref{diag:PeirceImbed}, we obtain the identifications
\begin{align*}
	H^0(X_k,&\linebundle|_{X_k})
	\cong H^0(\Peirce_k,\iota_k^*\linebundle)\\
	&\cong\Set{f:\Idem_k\to\CC\text{ hol.}}
	{\begin{aligned}
		&f(e,e') = \Delta(e,e'-c')f(c,c')\\
		&\text{for all Peirce equivalent $(e,e')$, $(c,c')$}
	\end{aligned}}\;.
\end{align*}
We note that due to Lemma~\ref{lem:JordanDeterminants}, $\Delta(e,e'-c') = \Delta_{\b e}'(c')$ for $\b e=(e,e')$, which is the (opposite) Jordan algebra determinant defined by the idempotent $\b e$. The main result of this section asserts, that the decomposition of $H^0(X,\linebundle)$ into $L$-types is obtained by restricting the sections to the closed $L$-orbits.
\begin{thm}\label{thm:Decomposition}
	The restriction map
	\begin{align}\label{eq:LTypeDecomposition}
		\rho: H^0(X,\linebundle)\to\bigoplus_{k=0}^r H^0(X_k,\linebundle|_{X_k}),\
		f\mapsto
		(f|_{X_0},\ldots,f|_{X_r})
	\end{align}
	is an isomorphism of $L$-modules.
\end{thm}
\begin{proof}
For convenience, we set $\sH:=H^0(X,\linebundle)$ and $\sH_k:=H^0(X_k,\linebundle|_{X_k})$. Since $X_k$ are closed $L$-orbits, each $\sH_k$ is an irreducible representation of $L$, and $\rho$ is an $L$-equivariant map. To show that $\rho$ is onto, we use an inductive argument on $k=0,\ldots, r$ for the map $\rho_k:\sH\to\sH_{r-k}\oplus\cdots\oplus\sH_r$. Consider the map $f(x,y):=\Delta(x,y-c')$ for $(x,y)\in V\times V'$ and some fixed $\b c=(c,c')\in\Idem$. Applying the basic identity $\Delta(u,v)\Delta(u^v,w) = \Delta(u,w+v)$ of the Jordan pair determinant, one immediately verifies that $f$ is an element of $\sH$. The restriction of $f$ to $\Idem_k$ is given by $f(e,e') = \Delta_\b e'(c')$ with $\b e = (e,e')$. Therefore, Lemma~\ref{lem:RankConditions} implies that $f|_{\Idem_k} = 0$ if and only if $\rank\b c < k$.
For $k=0$, choose $\b c\in\Idem$ with $\rank\b c = r$. Then, $\rho_0(f)$ is non-trivial in $\sH_r$, and since $\sH_r$ is irreducible and $\rho_0$ is $L$-equivariant, Schur's lemma implies that $\rho_0$ is onto. For $k>0$, choose $\b c\in\Idem$ with $\rank\b c = r-k$. Then, $\rho_k(f)$ is non-trivial only in the component of $\sH_{r-k}$. Therefore, the $L$-module generated by $\rho_k(f)$ in $\sH_{r-k}\oplus\cdots\oplus\sH_r$ is a non-trivial submodule of the first component, and by irreduciblibly of $\sH_{r-k}$, we obtain
\begin{align}\label{eq:SectionInclusion}
	\sH_{r-k}\oplus\{0\}\oplus\cdots\oplus\{0\}\subseteq\rho_{r-k}(\sH)\;.
\end{align}
Now let $(f_{r-k},\ldots,f_r)$ be any element in $\sH_{r-k}\oplus\cdots\oplus\sH_r$. By induction hypothesis, there exists an element $f\in\sH$ with $f|_{\Idem_\ell} = f_\ell$ for all $\ell>r-k$. Due to \eqref{eq:SectionInclusion}, the first component can be fixed by choosing a section $g\in\sH$ with $g|_{\Idem_{r-k}} = f_{r-k} - f|_{\Idem_{r-k}}$ and $g|_{\Idem_\ell} = 0$ for $\ell>r-k$. Therefore, $\rho_k(f+g) = (f_{r-k},\ldots,f_r)$, and hence $\rho_k$ is surjective.

To show that $\rho$ is injective, we have to show that a section $f\in H^0(X,\linebundle)$ that vanishes along all closed $L$-orbits $(X_k)_{k=0,\ldots,r}$ must also vanish on all of $X$. For this we use an inductive argument showing 
\begin{align}\label{eq:VanishingAssertion}
	f(e+c,e') = 0\quad
	\text{for all orthogonal $\b e = (e,e'),\;\b c = (c,c')\in \Idem$}\;.
\end{align}
From Proposition~\ref{prop:Representatives} and the fact that any element admits a completion to an idempotent, it follows that that any element in $V\times V'$ is projectively equivalent to some element of the form $(e+c,e')$ as in \eqref{eq:VanishingAssertion}. Therefore, showing \eqref{eq:VanishingAssertion}, also proves that $\rho$ is injecitve. As above, let $r$ be the rank of $(V,V')$. We prove \eqref{eq:VanishingAssertion} by induction on $n=r-k$ with $k = \rank(\b e)$. For $n=0$, the statement just reads $f(e,e')=0$ for all $\b e= (e,e')\in\Idem_r$, since $c\in V_0(\b e) =\{0\}$. This is satisfied by the assumption $f|_{\Idem_r} = 0$. For $n>0$, we start a second proof by induction, namely induction on $\rank(\b c)$. For $\rank(\b c) = 0$, i.e.,\ $c = 0$, we have $f(e,\,e') = 0$ due to the assumption that $f|_{\Idem_k}=0$. For $\rank(\b c)\geq 1$, consider the decomposition $\b c = \b c_1+\b c_2$ for some orthogonal idempotents $\b c_i =(c_i,c_i')$ with $\rank(\b c_2) = 1$. Due to the cocycle condition on $f$, we obtain for any $t\in\CC$
\[
	f(e + c_1+tc_2,\,e')
		= \Delta(e+c_1+tc_2,\,-(1-\tfrac{1}{t})c_2')\,f(e+c_1+c_2,e'+(1-\tfrac{1}{t})c_2')\;.
\]
By orthogonality of $\b e,\,\b c_1$ and $\b c_2$, the Jordan determinant simplifies to
\[
	\Delta(tc_2,\,-(1-\tfrac{1}{t})c_2')
		= \Delta(c_2,\,(1-t)c_2')
		= \Delta_c'(tc_2') = t^{\rank(\b c_2)} = t\;,
\]
and we therefore obtain
\[
	f(e + c_1+tc_2,\,e')
		= t\,g(\tfrac{1}{t})\quad\text{with}\quad
		g(\tfrac{1}{t}) := f(e+c_1+c_2,e'+(1-\tfrac{1}{t})c_2')
\]
for some entire function $g\in\mathcal O(\CC)$. Since the left hand side is holomorphic in $t$, it follows that $g(\tfrac{1}{t}) = \alpha + \beta\tfrac{1}{t}$ for some $\alpha,\beta\in\CC$. By induction hypothesis (on $n$), $g(0) = 0$, so $\alpha = 0$, and we conclude
\[
	f(e + c_1+tc_2,\,e') = \beta = \text{const.}
\]
Setting $t=0$, the induction hypothesis on $\rank(\b c)$ yields $\beta = f(u_+ + e_+,\,u_-) = 0$, and hence we also obtain $f(e+c,e') = f(e+c_2+c_1,\,e') = 0$. This finally proves \eqref{eq:VanishingAssertion}, and completes the proof of the theorem.
\end{proof}

The next proposition determines the highest weights of the $L$-types $H^0(X_j,\linebundle|_{X_j})$, and gives an explicit description of a highest weight vector. Let $(e_1,\ldots,e_r)$ be the frame of tripotents associated to the strongly orthogonal roots $\gamma_1,\ldots,\gamma_r$, cf.\ Section~\ref{subsec:Correspondence}, and let $\lambda$ be the fundamental weight associated to $\gamma_1$. Set $\epsilon_k:=\sum_{i=1}^k e_i$ for $k = 0,\ldots,r$.

\begin{prop}\label{P: niceterm}
	For $k\in\{0,\ldots,r\}$, the $L$-type $H^0(X_k,\linebundle|_{X_k})$ has highest weight (with 
	respect to $\Phi_L^+$)
	\[
		\lambda_k:= \lambda+\gamma_1+\cdots+\gamma_k\in (i\mf t)^*\;.
	\]
	In particular, the decomposition \eqref{eq:LTypeDecomposition} is multiplicity free. 
	Moreover, the map 
	\[
		f_k(x,y):=\Delta(x,y-\overline\epsilon_k)
	\]
	defines a section $f_k\in H^0(X,\linebundle)$, the restriction of which to $X_k$ is a highest 
	weight vector of $H^0(X_k,\linebundle|_{X_k})$. In local coordinates
	$\iota_0:V\hookrightarrow X$, the trivialization $\tilde f_k$ of $f_k$ is given by
	$\tilde f_k(x) = \Delta(x,-\overline\epsilon_k)$ for $x\in V$, and the restriction to the 
	$\CC$-linear span of the frame $(e_1,\ldots, e_r)$ is
	\[
		\tilde f_k(x) = (1+x_1)\cdots(1+x_k)\quad\text{for}\quad x = x_1e_1+\cdots+x_r e_r
	\] 
	with $x_1,\ldots,x_r\in\CC$.
\end{prop}
\begin{proof}
Let $Q_k$ be the $L$-stabilizer of the element $\GP{\epsilon_k}{\overline\epsilon_k}\in X_k$. According to Proposition~\ref{prop:LOrbits}, $Q_k$ is also the $L$-stabilizer of the Peirce 2-space $V_2(\epsilon_k)$, and Lemma~\ref{lem:BorelContained} below shows that the opposite Borel subgroup $B_L^\op$ (i.e., the one with Lie algebra $\mf b_L^\op:=\mf h\oplus\mf n_L^-$) is contained in $Q_k$. Therefore, to determine the weight of $H^0(X_k,\linebundle|_{X_k})$, it suffices to determine the character $\chi_k$ corresponding to the line bundle $\linebundle|_{X_k}$. Due to \eqref{diag:PeirceImbed}, this is the same character as the one corresponding to $\iota_k^*\linebundle$ on $\Peirce_k$. Hence, $\chi_k$ can be read off from the cocycle which describes sections in $\iota_k^*\linebundle$, namely
\[
	\chi_k:Q_k\to\CC,\ h\mapsto
		\chi(h)\cdot\Delta(\epsilon_k,\overline\epsilon_k - h^\#\overline\epsilon_k)
		= \chi(h)\cdot\Delta'_{\epsilon_k}(h^\#\overline\epsilon_k)\;,
\]
where $\chi$ is the character of $\linebundle$ (restricted to $Q_k\subseteq L$). We note that $\Delta'_{\epsilon_k}$ is the (opposite) Jordan algebra determinant, and one can show that $\Delta'_{\epsilon_k}(h^\#\overline\epsilon_k) = \Delta_{\epsilon_k}(h^{-1}\epsilon_k)$, which is a more common description. Here, we prefer to use the first formula for $\chi_k$. Let $\mf q_k$ denote the Lie algebra of $Q_k$. Next, we determine the derivative of $\chi_k$ at $e\in Q_k$ along an element $X\in\mf q_k$. By definition, the derivative of the character $\chi$ of $\linebundle$ is the linear functional $\lambda$. For the second term, $\tilde\chi_k(h):=\Delta(\epsilon_k,\overline\epsilon_k - h^\#\overline\epsilon_k)$, one shows (using standard properties of the Jordan pair determinant) that $\frac{\partial}{\partial y}\Delta(x,y)(z) = -\tfrac{1}{p}\,\Delta(x,y)\tau(x^y,z)$, where $p$ is defined by \eqref{eq:StructureConst}. By a short computation we thus obtain 
\[
	d\tilde\chi_k(e)(X) = \tfrac{1}{p}\,\tau(\epsilon_k,X^\#\overline\epsilon_k)
		= \tfrac{1}{p}\,\tau(X\epsilon_k,\overline\epsilon_k)\;.
\]
Recall that the Cartan subalgebra $\mf h\subseteq\mf q_k$ decomposes into $\mf h = \mf s_\CC\oplus\mf s_\CC^\bot$, where the elements $H\in\mf s_\CC^\bot$ have the property that $H(e_j) = 0$ for all $j$. Therefore, $d\tilde\chi_k(e)$ vanishes along $\mf s_\CC^\bot$. Moreover, $\mf s_\CC$ is spanned by the elements $H_{\gamma_j} = D_{e_j,\overline e_j}$ for $j = 1,\ldots, r$, and the central element $\Id_V$, and we obtain by (strong) orthogonality of the tripotents of the frame elements
\[
	d\tilde\chi_k(e)(H_{\gamma_j})
		= \tfrac{1}{p}\,\tau(D_{e_j,\overline e_j}\epsilon_k,\overline\epsilon_k)
		=\begin{cases}
			\tfrac{2}{p}\,\tau(e_j,\overline e_j) = 2 & ,\text{ if $j\leq k$,}\\
			0	&,\text{ if $j>k$.}
		\end{cases}
\]
Together with $d\chi(e)(\Id_V) = k$, it follows that $d\tilde\chi_k(e)$ coincides on $\mf h$ with the functional $\gamma_1+\cdots+\gamma_k$, and we therefore conclude that $d\chi_k(e) = \lambda + \gamma_1+\cdots+\gamma_k$. Next, we consider the map $f_k$.

Due to the identity $\Delta(u,v)\Delta(u^v,w) = \Delta(u,w+v)$ it follows that $f_k$ satisfies the cocycle condition for sections in $H^0(X,\linebundle)$. The evaluation of $\tilde f_k(x) = f_k(x,0)=\Delta(x,-\overline\epsilon_k)$ at $x=x_1e_1+\cdots+x_re_r$ is a special case of the
formula $\Delta(x,y) = \prod_{i=1}^r(1-x_iy_i)$ with $y = y_1\overline e_1+\cdots+y_r\overline e_r$, cf.\ \cite[\S16.15]{Lo75}. To prove that $f_k$ is a highest weight vector of the $L$-type $H^0(X_k,\linebundle|_{X_k})$ it suffices to show that $f_k$ is constant on an open subset of $X_k$ which contains $\GP{\epsilon_k}{\overline\epsilon_k}$, since $B_L^\op\subseteq Q_k$ (due to the next lemma). Indeed, $f_k$ is even constant on the open subset $U_{\overline\epsilon_k}=\Set{\GP{x}{\overline\epsilon_k}}{x\in V}\subseteq X$, since $f_k(x,\overline\epsilon_k) = \Delta(x,0) = 1$ for all $x\in V$. This completes the proof.
\end{proof}

\begin{lemma}\label{lem:BorelContained}
	Let $(e_1,\ldots,e_r)$ denote the frame of tripotents corresponding to the Cartan subalgebra 
	$\mf h$ of $\mf l$ as it is defined in Section~\ref{sec:PerlimLieTheo}. For $k=0,\ldots,r$ let 
	$\epsilon_k:=\sum_{i=1}^ke_i$. Then, the stabilizer subgroup
	$Q_k\subseteq L$ of the Peirce 2-space $V_2(\epsilon_k)$ contains the opposite Borel 
	subgroup $B_L^\op$.
\end{lemma}
\begin{proof}
This proof is based on the ideas of \cite{Up86}. Since the Peirce variety $\Peirce_k$ is a projective variety, the stabilizer subgroup $Q_k\subseteq L$ is parabolic, and hence it suffices to prove this lemma on the Lie algebra level, i.e., we show that $\mf b_L^\op\subseteq\mf q_k$ with $\mf q_k=\Set{T\in\mf l}{T V_2(\epsilon_k)\subseteq V_2(\epsilon_k)}$. Let $V = \bigoplus_{0\leq i\leq j\leq r} V_{ij}$ be the joint Peirce decomposition with respect to the frame of tripotents $(e_1,\ldots, e_r)$, set $V_{ij}:=V_{ji}$ for $i\neq j$, and define
\[
	\mf l_{ij}:=\Set{D_{u,\overline v}\in\mf l}
									{u\in V_{i\ell},\ v\in V_{\ell j}\text{ for some $\ell$}}\;.
\]
We note that $\mf l_{ij}\neq\mf l_{ji}$ for $i\neq j$. We claim that
\[
	\mf l = \mf l_0\oplus\bigoplus_{i\neq j}\mf l_{ij}\quad\text{with}\quad 
	\mf l_0:=\bigoplus_{i=1}^r\mf l_{ii}
\]
is the weight space decomposition of $\mf l$ with respect to the adjoint action of the abelian subalgebra $\mf s=\left\langle i\,D_{e_j,\overline e_j},\; i\Id_V\,\mid\,j=1,\ldots,r\right\rangle_\RR\subseteq\mf t$. Indeed, since $\mf l=[\mf p^+,\mf p^-]$, $\mf l$ is generated by all derivations of the form $D_{u,\overline v}$ with $u,v\in V$, and for $T=D_{u,\overline v}\in\mf l_{ij}$, the Jordan identity \JP{15} implies
\[
	[D_{e_k,\overline e_k},D_{u,\overline v}]
		= D_{\JTP{e_k}{\overline e_k}{u},\overline v} -D_{u,\JTP{\overline e_k}{e_k}{\overline v}}
		=(\delta_{ki}-\delta_{kj})\,D_{u,\overline v}\;,
\]
where $\delta_{ij}$ is Kronecker's delta. Moreover, $i\Id_V$ acts trivially on each $\mf l_{ij}$. With \eqref{eq:GammaIdentity}, it readily follows that the weight corresponding to $\mf l_{ij}$ is $(\gamma_i-\gamma_j)/2$, where we also set $\gamma_0:=0$.

Now let $\mf l = \mf t\oplus\bigoplus_{\alpha\in\Phi_L}\mf l_\alpha$ be the root decomposition of $\mf l$ with respect to the maximal torus $\mf t$. By \cite[Lemma 13]{HC56}, a positive root $\alpha\in\Phi_L^+$ either vanishes on $\mf s$, or its restriction to $\mf s$ is of the form $(\gamma_i-\gamma_j)/2$ with $1\leq j< i\leq r$ or $-\gamma_i/2$. We thus obtain that 
\[
	\mf p:=\bigoplus_{1\leq i\leq j\leq r}\mf l_{ij}\oplus\bigoplus_{i=1}^r\mf l_{i0}
\]
is a subalgebra of $\mf l$ which contains the opposite Borel subgalgebra $\mf b_L^{\textup{op}}$.

In the last step of the proof, we show that $\mf p\subseteq\mf q_k$. Due to the derivation property of elements $T\in\mf p$, the relation $T\epsilon_k\in V_2(\epsilon_k)$ already implies $T V_2(\epsilon_k)\subseteq V_2(\epsilon_k)$. First assume that $T\in\mf l_{ii}$, so $T = D_{u,\overline v}$ with $u,v\in V_{i\ell}$. Since $V_{i\ell}$ is contained in one of the Peirce spaces $V_m(\epsilon_k)$ with $m\in\{0,1,2\}$, the Peirce rules imply $D_{u,\overline v}\epsilon_k\in V_2(\epsilon_k)$. Now assume $T\in\mf l_{ij}$ for $1\leq i<j\leq r$. If $i\leq k$, then the joint Peirce rules imply $D_{u,\overline v}\epsilon_k\in V_{ik}$, which is a subspace of $V_2(\epsilon_k)$. If $i>k$, then $j>k$ and the joint Peirce rules imply $D_{u,\overline k}\epsilon_k = 0$. For the last case, $T\in\mf l_{i0}$, i.e.\ $T =D_{u,\overline v}$ with $u\in V_{i\ell}$, $v\in V_{\ell0}$, the joint Peirce rules yield $D_{u,\overline v}\epsilon_k = 0$. To sum up, this shows that $T\epsilon_k\in V_2(\epsilon_k)$ for all $T\in\mf p$, which completes the proof.
\end{proof}

\section{Reduced spaces for the $K$-action}
\subsection{Moment map}
In this section we explicitely determine the moment map of the $K$-action on the compact Hermitian symmetric space $X=U/K$. Recall from Section~\ref{subsec:PrequantumBundle} that $X$ can be realized as the coadjoint orbit $\mathcal O_\lambda\subseteq\mf u^*$ with base point $i\lambda\in\mf u^*$, where $-\lambda$ is the extension of the fundamental weight associated to $\gamma_1$. In this realization the moment map $\mu:\mathcal O_\lambda\to\mf u^*$ is just the restriction of the identity map to $\mathcal O_\lambda$, and the moment map $\mu_\mf k$ corresponding to the $K$-action on $X\cong\mathcal O_\lambda$ is given by restriction to $\mf k$, i.e., $\mu_\mf k(x) = \mu(x)|_{\mf k}$ for all $x\in\mathcal O_\lambda$.

In the following we use a constant multiple of the Killing form $\kappa$ on $\mf u$ to identify $\mf u^*$ with $\mf u$, $\mf k^*$ with $\mf k$, and coadjoint orbits with the corresponding adjoint orbits. More precisely, we use
\begin{align}\label{eq:DualIdentification}
	\vartheta:\mf u\to\mf u^*,\ X\mapsto -\tfrac{1}{2p}\,\kappa(X,\Box)\;,
\end{align}
where $p$ is defined in \eqref{eq:StructureConst}. Let $\tilde\mu:=\vartheta\circ\mu$ and $\tilde\mu_\mf k:=\vartheta\circ\mu_\mf k$ denote the corresponding moment maps. Then, $\tilde\mu_\mf k = \proj_\mf k\circ\tilde\mu$, where $\proj_\mf k$ is the orthogonal projection of $\mf u$ onto $\mf k$ with respect to the Killing form.

\begin{lemma}\label{lem:DualFundamWeight}
	The map $\vartheta$ identifies $i\lambda\in\mf u^*$ with $i\Id_V\in\mf k$, which is central in
	$\mf k$.
\end{lemma}
\begin{proof}
Let $\eta\in\mf u$ be the unique element determined by the relation $i\lambda(Y) = \kappa(X,Y)$ for all $Y\in\mf u$. We have to show that $X = -\tfrac{i}{2p}\,\Id_V$. Since $\lambda$ vanishes on $\mf k^\bot$, it follows that $X\in\mf k$. In addition, the relation $(\Ad_u^*\lambda)(Y) = \kappa(X,\Ad_u^{-1}Y) = \kappa(\Ad_uX,Y)$ implies that the stabilizer of the coadjoint action of $U$ with respect to $\lambda$ (which is $K$) coincides with the stabilizer of the adjoint action of $U$ with respect to $X$. Therefore, $X$ is central in $\mf k$, and since $\mf z(\mf k) = \RR\cdot i\Id_V$, we conclude that $X = c\cdot i\Id_V$ for some $c\in\RR$. Finally, using \eqref{eq:KillingForm}, the relation $\lambda(H_{\gamma_1}) = -1$ with $H_{\gamma_1}=D_{e_1,\overline e_1}$ implies
\begin{align*}
	-i&=(i\lambda)(H_{\gamma_1}) = \kappa(c\cdot i\Id_V,D_{e_1,\overline e_1})
		=2i\,c\cdot\Tr D_{e_1,\overline e_1}
		=2i\,c\cdot\tau(e_1,\overline e_1)\;,
\end{align*}
and since $p=\tau(e_1,\overline e_1)$, this completes the proof.
\end{proof}
As a first step, we determine the moment map on the open and dense subset $V\subseteq X$. Recall the Jordan-theoretic description of $X$ via projective equivalence, $X\cong(V\times V')/R$, so elements of $X$ are equivalence classes $\GP{x}{a}$ of elements in $V\times V'$. The embedding $\iota_0:V\hookrightarrow X$ is given by $\iota_0(x)=\GP{x}{0}$. The moment map on $X=U/K$ is by definition identified with the moment map $\tilde\mu_\mf k$ on $\vartheta^{-1}(\mathcal O_\lambda)\subseteq\mf u$ via the isomorphism given by $uK\mapsto\Ad_{u}(i\Id_V)$.

\begin{prop}\label{prop:ExplicitMomentMap1}
	The restriction of the moment map $\tilde\mu_\mf k:X\to\mf k$ of the $K$-action on 
	$X$ to $V\subseteq X$ is given by
	\begin{align}\label{eq:MomentMapFormula1}
		\tilde\mu_\mf k: V\to\mf k,\; x \mapsto 
			i\,\big(\B{x}{-\overline x}^{-1}-Q_x\B{-\overline x}{x}^{-1}Q_{\overline x}\big)\;.
	\end{align}
	If $x = \sum_{j=1}^k\sigma_je_j$ is the spectral decomposition of $x\in V$, then
	\begin{align}\label{eq:MomentMapFormula2}
		\tilde\mu_\mf k(x) = i\left(\Id_V - 
									 \sum_{j=1}^k \tfrac{\sigma_j^2}{1+\sigma_j^2}\,D_{e_j,\overline e_j}\right)\,.
	\end{align}
\end{prop}
\begin{proof}
Let $\pi:U\to X\cong U/K$ be the canonical projection of $U$ onto $X$ given by $u\mapsto u\GP{0}{0}$, then the embedding $\iota_0:V\hookrightarrow X$ of $V$ into $X$ admits a lift to $U$ given by
\[
	\varphi:V\to U,\; x\mapsto
		u_x:=\trans{x}\circ\B{x}{-\overline x}^{\half}\circ\qtrans{\overline x}\;,
\]
Indeed, since $\qtrans{\overline x}\GP{0}{0} = \GP{0}{\overline x} = \GP{0}{0}$, we obtain
\[
	\pi\circ\varphi(x)
		= u_x\GP{0}{0}
		= \trans{x}\circ\B{x}{-\overline x}^{\half}\GP{0}{0}
		= \trans{x}\GP{0}{0}
		= \GP{x}{0}
		=\iota_0(x)\;.
\]
Therefore, the moment map on $V\subseteq X$ is given by $\tilde\mu:V\to\mf u$ with $\tilde\mu^x:=\tilde\mu(x)=\Ad_{u_x}(i\Id_V)$. Recall that all Lie algebras are realized as vector fields on $V$, and that the adjoint action reads
\[
	(\Ad_{u^{-1}}\zeta)(z) = du(z)^{-1}\cdot\zeta(u(z))\;.
\]
Explicitly, since $u_x^{-1}(z) = \trans{-x}\B{x}{-\overline x}^\half\qtrans{-\overline x}(z) = \B{x}{-\overline x}^\half z^{-\overline x} - x$, we obtain
\begin{align*}
	\tilde\mu^x(z)
		&= i\,\big(\B{x}{-\overline x}^\half
							 \B{z}{-\overline x}^{-1}\big)^{-1}\cdot\big(i\,u_x^{-1}(z)\big)\\
		&= i\,\B{z}{-\bar x} z^{-\bar x} - i\, \B{z}{-\bar x}\B{x}{-\bar x}^{-\half} x \\
		&= i\big(z+Q_z\bar x - \B{z}{-\bar x}(x^{-\bar x})\big)\;.
\end{align*}
In the last step we used the identity $\B{x}{-\bar x}^{-\half} x = x^{-\bar x}$, which is obvious when $x$ is replaced by its spectral decomposition $x = \sum\sigma_j e_j$. Recall that the moment map $\tilde\mu_\mf k$ of the $K$-action on $V\subseteq X$ is given by the orthogonal projection of $\tilde\mu^x$ onto $\mf k$. The orthogonal projection yields the linear terms of the vector field $\tilde\mu^x$, therefore,
\begin{align*}
	\tilde\mu_\mf k^x(z) = d\tilde\mu^x(0)\cdot z = i\,(z-D_{z,\bar x}(x^{-\bar x}))
		= i\,(\Id_V - D_{x^{-\bar x},\,\bar x})(z)\;.
\end{align*}
From this, formula \eqref{eq:MomentMapFormula1} follows with a short calculation by using the relations $\B{x}{y}D_{x^y,z} = D_{x,z}-Q_xQ_{y,z}$ (\JP{30}) and $\B{x}{y}Q_{x^y} = Q_{x^y}\B{y}{x} = Q_x$ (\JP{28}), and formula \eqref{eq:MomentMapFormula2} follows by using the spectral decomposition $x = \sum\sigma_je_j$. This completes the proof.
\end{proof}

For the extension of the result of Proposition~\ref{prop:ExplicitMomentMap1} to all of $X$, we need the following operator.

\begin{lemma}\label{lem:GammaOperator}
	For $(x,a)\in(V,V')$, let
	$\Gamma_{x,a}:=\B{x}{a}\B{x^a}{-{\overline x}^{\overline a}}\B{\overline a}{\overline x}$.
	Then
	\begin{enumerate}\renewcommand{\labelenumi}{\normalfont(\alph{enumi})}
		\item $\Gamma_{x,a}$ depends polynomially on $x,\overline x,a,\overline a$.
		\item If $\GP{x}{a} = \GP{z}{b}$, then
					$\Gamma_{z,b}
						= \B{x}{a-b}^{-1}\Gamma_{x,a}\B{\overline a-\overline b}{\overline x}^{-1}$.
		\item $\Gamma_{x,a}$ is a positive definite operator on $V$ with respect to the
					inner product $(-|-)$ defined in \eqref{eq:innerproduct}.
	\end{enumerate}
	Mutatis mutandis, the same results hold for the adjoint operator 
	$\Gamma_{\overline x,\overline a}:=\Gamma_{x,a}^\#
		= \B{\overline x}{\overline a}\B{{\overline x}^{\overline a}}{-x^a}\B{a}{x}\in\End(V')$
	where the adjoint is taken with respect to the trace form $\tau$.
\end{lemma}
\begin{proof}
The proof of (a) and (b) relies on the fact that the Bergman operator satisfies the relations $\B{u}{v}\B{u^v}{w} = \B{u}{v+w}$ and $\B{w}{u^v}\B{v}{u} = \B{v+w}{u}$. For (a), consider the operator $\Gamma'_{x,a,y,b}:=\B{x}{a}\B{x^a}{-{y}^{b}}\B{b}{y}$ with $(x,a)\in(V,V')$ and $(y,b)\in (V',V)$. Then
\[
	\Gamma'_{x,a,y,b}=\B{x}{-{y}^{b}+a}\B{b}{y}
		=\B{x}{a}\B{-x^a+b}{{y}^{b}}\;,
\]
and the first of these identities implies that $\Gamma'_{x,a,y,b}$ is polynomial in $(x,a)\in V$, whereas the second identity implies that $\Gamma'_{x,a,y,b}$ is polynomial in $(y,b)\in(V',V)$. Therefore, $\Gamma_{x,a} =\Gamma'_{x,a,\overline x,\overline a}$ is polynomial in $x,a,\overline x,\overline a$. Part (b) is a simple application of the identities statisfied by the Bergman operators, since $\GP{z}{b} = \GP{x}{a}$ implies $z = x^{a-b}$. For (c), we first note that $\Gamma_{x,a}$ is self-adjoint with respect to $(-|-)$. Due to (b) and Proposition~\ref{prop:Representatives}, it suffices to prove (c) for $(x,a) = (e+z,\overline e)$ for some tripotent $e\in\Tri$ and $z\in V_0(e)$. We calculate $\Gamma_{e+z,\overline e}$ by means of the limit $\lim_{t\to 1}\Gamma_{e+z,\overline{te}}$ with $t\in\RR\setminus\{1\}$. Using the relation $(e+z)^{\overline{te}} = \frac{1}{1-t}\, e + z$ and the Peirce rules, it is straightforward to obtain
\begin{align}\label{eq:GammaOperator}
	\Gamma_{e+z,\overline e} = \lim_{t\to 1}\Gamma_{e+z,\overline{te}}
		= \lim_{t\to 1}\B{(1-t)e}{-(1-t)\overline e}\B{z}{-\overline z}
		= \B{z}{-\overline z}\;.
\end{align}
Let $z = \sum_{i=1}^k\sigma_ie_i$ be the spectral decomposition of $z$, and let $V = \bigoplus_{0\leq i\leq j\leq k}V_{ij}$ be the joint Peirce decomposition with respect to the orthogonal system of tripotents $(e_1,\ldots,e_k)$. Then \cite[\S3.15]{Lo77}, the relation $\B{z}{-\overline z}v_{ij} = (1+\sigma_i^2)(1+\sigma_j^2)v_{ij}$ holds for all $v_{ij}\in V_{ij}$, where $\lambda_0:=0$. This shows that $\B{z}{-\overline z}$ and hence $\Gamma_{e+z,\overline z}$ is a positive definite operator on $V$.
\end{proof}

\begin{thm}\label{thm:ExplicitMomentMap}
	The moment map $\tilde\mu_\mf k:X\to\mf k$ of the $K$-action on $X$ is given by
	\begin{align}\label{eq:MomentMapFormula3}
		\tilde\mu_\mf k(\GP{x}{a})
			= i\,\big(\B{\overline a}{\overline x}\Gamma_{x,a}^{-1}\B{x}{a}
				- Q_x\Gamma_{\overline x,\overline a}^{-1}Q_{\overline x}\big)
	\end{align} 
	If $(e+z,\overline e)\in\GP{x}{a}$ is a representative as in 
	Proposition~\ref{prop:Representatives}, and $z=\sum_{j=1}^k\sigma_je_j$ is the spectral 
	decomposition of $z$, then
	\begin{align}\label{eq:MomentMapFormula4}
		\tilde\mu_\mf k(\GP{x}{a})
			= i\left(\Id_V - 
				 \sum_{j=1}^k \tfrac{\sigma_j^2}{1+\sigma_j^2}\,D_{e_j,\overline e_j}
				 - D_{e,\bar e}\right)\;.
	\end{align}
\end{thm}
\begin{proof}
For \eqref{eq:MomentMapFormula3}, it suffices to check that $\tilde\mu_\mf k(\GP{x}{a})$ is well-defined and that its restriction to $V\subseteq X$ coincides with \eqref{eq:MomentMapFormula1}. The latter part is easy to see, since for $a=0$, we obtain $\B{x}{0} = \Id_V$ and $\Gamma_{x,0} = \B{x}{-\bar x}$. To prove that $\tilde\mu_\mf k(\GP{x}{a})$ is well-defined, let $\GP{x}{a} = \GP{z}{b}$, so $z = x^{a-b}$. Due to Lemma~\ref{lem:GammaOperator}, $\Gamma_{z,b} = \B{x}{a-b}^{-1}\Gamma_{x,a}\B{\overline a-\overline b}{\overline x}^{-1}$, and hence
\begin{align*}
	\tilde\mu_\mf k(\GP{z}{b})
			&= i\,\big(\B{\overline b}{\overline z}\Gamma_{z,b}^{-1}\B{z}{b}
				- Q_z\Gamma_{\overline z,\overline b}^{-1}Q_{\overline z}\big)\\
			&= i\,\big(\B{\overline b}{{\overline x}^{\overline a-\overline b}}
								 \B{\overline a-\overline b}{\overline x}\Gamma_{x,a}^{-1}
								 \B{x}{a-b}\B{x^{a-b}}{b}
							 + Q_{x^{a-b}}
								 \B{a-b}{x}\Gamma_{\overline x,\overline a}^{-1}
								 \B{\overline x}{\overline a-\overline b}
								 Q_{{\overline x}^{\overline a-\overline b}}\big)\\
			&=i\,\big(\B{\overline a}{\overline x}\Gamma_{x,a}^{-1}\B{x}{a}
				- Q_x\Gamma_{\overline x,\overline a}^{-1}Q_{\overline x}\big)\\
			&=\tilde\mu_\mf k(\GP{x}{a})\;.
\end{align*}
We note that continuity also yields the condition $\tilde\mu_\mf k(\GP{x}{a})\in\mf k$. The identity \eqref{eq:MomentMapFormula4} follows by continuity from \eqref{eq:MomentMapFormula2}, since $\GP{x}{a} = \GP{e+z}{\overline e} = \lim_{t\to 1}\GP{\tfrac{1}{1-t}e + z}{0}$, and the spectral decomposition of $\tfrac{1}{1-t}e + z$ is just $\tfrac{1}{1-t}e + z = \tfrac{1}{1-t}e + \sum\sigma_je_j$, since $z\in V_0(e)$.
\end{proof}

\begin{cor}\label{cor:MomentMapImage}
	The image of the moment map $\tilde\mu_\mf k$ in $\mf k$ is given by
	\[
		\tilde\mu_\mf k(X) = \Set{i\,\big(\Id_V - \sum_{i=1}^r\nu_i D_{e_i,\overline e_i}\big)}
			{\begin{aligned}
				&(e_1,\ldots,e_r)\text{ frame of tripotents,} \\
				&(\nu_1,\ldots,\nu_r)\in[0,1]^r
			 \end{aligned}}\;.
	\]
\end{cor}

In the last part of this section, we return to the moment map $\mu_\mf k$ in its original form with image in $\mf k^*$, and determine the intersection of $\mu_\mf k(X)$ with the closed positive Weyl chamber $i\mf c\subseteq \mf k^*$ (with respect to $\Phi_L^+$), i.e.,
\[
	\mf c := \Set{\alpha\in(i\mf t)^*}{\alpha(H_\beta)\geq 0\text{ for all $\beta\in\Phi_L^+$}}\;.
\]
Recall from Section~\ref{subsec:PrequantumBundle} that real valued functionals $\alpha\in(i\mf t)^*$ are identified with their complex extensions to functionals on $\mf l$ (zero-extension on the orthogonal complement of $i\mf t$), and the restriction of $i\alpha$ to $\mf k$ is a real valued functional on $\mf k$, so $i\mf c\subseteq\mf k^*$.

\begin{thm}
	The intersection of the image of the moment map $\mu_\mf k$ with $i\mf c$, is given by
	\[
		\mu_\mf k(X)\cap i\mf c
			= \Set{i(\lambda + \sum_{j=1}^r\nu_j\gamma_j)\in\mf t^*}
						{1\geq\nu_1\geq\ldots\geq\nu_r\geq 0}
			=:\Pi_\mf s\;.
	\]
	In particular, this is a convex polytope.
\end{thm}
\begin{proof}
Let $(e_1,\ldots,e_r)$ be the frame of tripotents that is associated to the maximal abelian subalgebra $\mf t\subseteq\mf u$ as described in Section~\ref{subsec:Correspondence}. Consider
\[
	\tilde\Pi_\mf s:=\Set{i\big(\Id_V-\sum\nolimits_{i=1}^r\nu_iD_{e_i,\overline e_i}\big)}
											 {1\geq\nu_1\geq\ldots\geq\nu_r\geq 0}\subseteq\mf k\;.
\]
We first show that $\vartheta(\tilde\Pi_\mf s)=\Pi_\mf s$ with $\vartheta$ as in \eqref{eq:DualIdentification}. Indeed, Lemma~\ref{lem:DualFundamWeight} shows that $\vartheta(i\Id) = i\lambda$, and using \eqref{eq:GammaIdentity}, we obtain for all $T\in\mf k$,
\begin{align*}
	\vartheta(iD_{e_j,\overline e_j})(T)
		&= -\tfrac{i}{2p}\,\kappa(D_{e_j,\overline e_j},T)
		 = \tfrac{i}{2p}\,\kappa(T,[e_j,q_{\overline e_j}])\\
		&= \tfrac{i}{2p}\,\kappa(Te_j,q_{\overline e_j})
		 = -\tfrac{i}{p}\,\tau(Te_j,\overline e_j)
		 = -i\,\gamma_j(T)\;.
\end{align*}
Therefore, $\vartheta(iD_{e_j,\overline e_j}) = -i\gamma_j$, and we conclude that $\vartheta(\tilde\Pi_\mf s) = \Pi_\mf s$. Since $K$ acts transitively on the set of frames of tripotents, Corollary~\ref{cor:MomentMapImage} implies that $Ad_K(\tilde\Pi_\mf s) = \tilde\mu_\mf k(X)$, and since $\vartheta$ is $K$-equivariant, this also yields
\begin{align}\label{eq:WeylIntersection1}
	\Ad_K^*(\Pi_\mf s) = \mu_\mf k(X)\;.
\end{align}
As a second step, we prove that $\Pi_\mf s\subseteq i\mf c$: Theorem~\ref{thm:Decomposition} and Proposition~\ref{P: niceterm} imply that for all $\ell\in\{0,1,\ldots,r\}$ the functional $\lambda_\ell:=\lambda+\sum_{j=1}^\ell\gamma_j\in(i\mf t)^*$ is a highest weight of an $L$-type in $H^0(X,\linebundle)$, hence $i\lambda_\ell\in \Pi_\mf s\cap i\mf c$. Since $i\mf c$ is convex, we conclude that
\begin{align}\label{eq:WeylIntersection2}
	\Pi_\mf s = \textup{conv}\{i\lambda_0,\ldots,i\lambda_r\}\subseteq i\mf c\;.
\end{align}
Finally, since the closed Weyl chamber $i\mf c$ is a fundamental domain for the $K$-action on $\mf k^*$ (cf.\ \cite[Lemma~3.8.2]{DK00}), it follows from \eqref{eq:WeylIntersection1} and \eqref{eq:WeylIntersection2} that $\mu_\mf k(X)\cap i\mf c = \Pi_\mf s$.
\end{proof}

\subsection{Reduced spaces}
The next goal is to show that the reduced spaces are points, so we first determine the fibre $\tilde\mu_\mf k^{-1}(T)$ of a given element $T = i\,(\Id_V - \sum_{j = 1}^k\nu_j D_{e_j,\overline e_j})\in\mf k$.

\begin{lemma}\label{lem:OperatorEquality}
	Let $(e_1,\ldots,e_k)$ and $(c_1,\ldots, c_\ell)$ be two systems of orthogonal tripotents, and 
	let $\nu_1<\nu_2<\cdots <\nu_k$ and $\mu_1<\mu_2<\cdots<\mu_\ell$ be non-zero real numbers. Then
	\[
		\sum_{i=1}^k\nu_i D_{e_i,\overline e_i} = \sum_{j=1}^\ell\mu_j D_{c_j,\overline c_j}
		\ \iff\ 
		k =\ell\text{ and } \nu_i = \mu_i,\ e_i\approx c_i\text{ for all $i$.}
	\]
	Here, $e_i\approx c_i$ means that $e_i$ and $c_i$ induce the same Peirce decompositions, i.e., 
	$V_m(e_i) = V_m(c_i)$ for $m\in\{0,1,2\}$.
\end{lemma}
\begin{proof}
For convenience, we set $A:=\sum_{i=1}^k\nu_i D_{e_i,\bar e_i}$ and $B:=\sum_{j=1}^\ell\mu_j D_{c_j,\bar c_j}$. Let $V = \bigoplus_{0\leq i\leq j\leq r} V_{ij}$ be the joint Peirce decomposition with respect to the orthogonal family $(e_1,\ldots,e_k)$. Then, the Peirce rules imply $Ax_{ij} = (\nu_i+\nu_j)\,x_{ij}$ for $x_{ij}\in V_{ij}$, where we also set $\nu_0:=0$. Therefore, $V$ decomposes into eigenspaces of $A$, and the eigenspace of the highest eigenvalue, namely $2\nu_k$, is $V_2(e_k)$. In the same way we obtain a decomposition of $V$ into eigenspaces of $B$, and the eigenspace of the highest eigenvalue $2\mu_\ell$ is $V_2(c_\ell)$. Assuming $A=B$, we therefore obtain $\nu_k = \mu_\ell$ and $V_2(e_k) = V_2(c_\ell)$. Since Peirce $2$-spaces corresponding to tripotents uniquely determine the whole Peirce decomposition, this also implies $V_m(e_k) = V_m(c_\ell)$ for $m\in\{0,1,2\}$, and hence $D_{e_k,\bar e_k} = D_{c_\ell,\bar c_\ell}$. Therefore, the assumption $A=B$ is reduced to $\sum_{i=1}^{k-1}\nu_i D_{e_i,\bar e_i} = \sum_{j=1}^{\ell-1}\mu_j D_{c_j,\bar c_j}$ and the statement follows by induction.
\end{proof}

\begin{thm}[Reduced spaces]\label{T: reducedspaces}
	Let $\tilde\mu_\mf k:X\to\mf k$ be the moment map of the $K$-action on $X$, and let
	$T\in\tilde\mu_\mf k(X)$, i.e., $T=\tilde\mu_\mf k(\GP{e_0+z}{\overline e_0})$ for some 
	tripotent $e_0\in\Tri$ and $z\in V_0(e_0)$. Let $z=\sum_{j=1}^k\sigma_je_j$ be the spectral 
	decomposition of $z$.
	\begin{enumerate}\renewcommand{\labelenumi}{\normalfont(\alph{enumi})}
		\item The fibre of $T$ with respect to the moment map is given by
					\begin{align*}
						\tilde\mu_\mf k^{-1}(T)
						= \Set{\GP{c_0 + \sum_{j=1}^k\sigma_j\,c_j}{\overline c_0}}
							{\begin{aligned}
								&c_j\in\Tri,\ c_j\approx e_j \\
								&\text{for all $j=0,\ldots, k$}
							 \end{aligned}}\;.
					\end{align*}
		\item The stabilizer subgroup $K_T$ of $T$ consists of those elements leaving the Peirce spaces
					of the joint Peirce decomposition $V=\bigoplus V_{ij}$ with 
					respect to the orthogonal system $(e_1,\ldots,e_k,e)$ invariant.
		\item The reduced space $\tilde\mu_\mf k^{-1}(T)/K_T$ is a point.
	\end{enumerate}
\end{thm}
\begin{proof}
Recall from Proposition~\ref{prop:Representatives} that any element of $X$ can be represented as $\GP{c_0+w}{\overline c_0}$ for some tripotent $c_0$ and $w\in V_0(c_0)$. Let $w = \sum_{i=1}^\ell\tau_ic_i$ be the spectral decomposition of $w$. Now assume that $\tilde\mu_\mf k(\GP{e_0+z}{\overline e_0}) = \tilde\mu_\mf k(\GP{c_0+w}{\overline c_0})$. Then, Theorem~\ref{thm:ExplicitMomentMap} and Lemma~\ref{lem:OperatorEquality} imply that $k=\ell$ and $\sigma_j^2/(1+\sigma_j^2) = \tau_j^2/(1+\tau_j^2)$, $c_j\approx e_j$ for all $j=0,\ldots, k$. Moreover, since the identity between $\sigma_j$ and $\tau_j$ is solved only for $\sigma_j=\pm\tau_j$, and $\sigma_j$, $\tau_j$ are assumed to be positive, this proves (a). For (b), we note that $\Ad_k D_{e_i,\overline e_i} = D_{ke_j,\overline{ke}_j}$, and hence Lemma~\ref{lem:OperatorEquality} implies that $k\in K$ stabilizes $T$ if and only if $e_j\approx k e_j$ for all $j$. Since the Peirce spaces of the joint Peirce decomposition corresponding to $(e_0,e_1,\ldots,e_k)$ can be described by intersections of the Peirce spaces $V_m(e_j)$ with $m\in\{0,1,2\}$, $j=0,\ldots,k$, and, conversely, the Peirce spaces $V_m(e_j)$ are given by direct sums of joint Peirce spaces. This proves (b). For (c), we have to show that $K_T$ acts transitively on $\tilde\mu_\mf k^{-1}(T)$. Due to (a), this is equivalent to the statement that $K_T$ conjugates any orthogonal systems $(e_0,\ldots,e_k)$ and $(c_0,\ldots,c_k)$ of tripotents with $e_j\approx c_j$ for all $j$. This follows from the fact \cite[\S5.9]{Lo77} that $K$ acts transitively on the set of frames of tripotents: since each $e_j$ and $c_j$ can be decomposed further into orthogonal primitive tripotents to obtain frames, and since $\rank e_j = \rank c_j$, there exists an element $k\in K$ mapping $c_j$ onto $e_j$ for all $j$. Since $c_j\approx e_j$, $k$ preserves the Peirce spaces $V_m(e_j) = V_m(c_j)$. By the same argument as for (b), it follows that $k\in K_T$.
\end{proof}

\section{Branching laws}
Before we turn to the problem of decomposing the 
spaces $H^0(X,\mathscr{L}^k)$ under the group $K$ 
we recall some notions from the theory of Hamiltonian 
actions of compact Lie groups. For a more thorough 
treatment, we refer to \cite{GS82} and \cite{S95}.

Let $K$ temporarily denote an arbitrary compact 
connected Lie group which acts 
holomorphically on the connected compact 
K\"ahler manifold $(M, \Omega)$ in a Hamiltonian fashion, 
and let $\tau: M \rightarrow \fk^*$ be 
the moment map for the action. 
Assume that $\mathcal{L} \rightarrow M$ is a 
prequantum line bundle for $(M, \Omega)$. Then 
$K$ acts holomorphically on $\mathcal{L}$ as bundle 
isomorphisms, and this action extends to an action 
of the universal complexification, $K^{\C}$, of 
$K$. We recall two notions of stability for the 
$K$-action on $M$. 
Firstly, we have the set 
\begin{equation*}
M_{ss}(\mathcal{L}):=\{m \in M \mid s(m) \neq 0 \,\,\mbox{for some $k \in \N$, and 
$s \in H^0(M, \mathcal{L}^k)^K$}\}
\end{equation*}
of \emph{algebraically semistable} points. 
Secondly, we have the set
$$M_{ss}=\{m \in M \mid \overline{K^\C.m} \cap 
\tau^{-1}(0) \neq \emptyset \}$$
of \emph{analytically semistable} points. 
In fact, the identity
\begin{equation*}
M_{ss}(\mathcal{L})=M_{ss} \label{E: stabeq}
\end{equation*} 
holds (cf. \cite[Thm. 2.18]{S95}).

We now turn to our particular setting, where 
$K$ again is the stabilizer in $U$ of $eQ\in X$.  
Let $\mathcal{O}_\xi^K$ be a coadjoint orbit 
in $\fk^*$ through an integral $i\xi \in \ft^* \subseteq \fk^*$, 
which we view as a symplectic manifold when equipped 
with the Kostant-Kirillov symplectic form $\omega^K_\xi$. 
Let $\overline{\mathcal{O}^K_\xi}$ denote 
$\mathcal{O}^K_\xi$ equipped with the symplectic form 
$-\omega^K_\xi$, and with the reverse 
complex structure, i.e., the sheaf of 
holomorphic functions on $\overline{\mathcal{O}^K_\xi}$
is the sheaf of antiholomorphic functions on 
$\mathcal{O}^K_\xi$. If $\mathscr{L}_\xi$ is 
the prequantum line bundle for $\mathcal{O}^K_\xi$, let
$\overline{\mathscr{L}_\xi^*}$ denote the 
line bundle over $\mathcal{O}^K_\xi$ where 
the fibre over $x$ is given by the 
antilinear functionals $(\mathscr{L}_\xi)_x \rightarrow \C$.
Then $\overline{\mathscr{L}_\xi^*}$ is a holomorphic 
prequantum line bundle for $(\overline{\mathcal{O}^K_\xi},-\omega_\xi)$.

Consider now the product space 
$X \times \overline{\mathcal{O}^K_\xi}$. 
Let $p_1$ and $p_2$ denote the projections onto 
$X$ and $\overline{\mathcal{O}^K_\xi}$, respectively. 
For $k \in \N$, we equip $X \times \overline{\mathcal{O}^K_\xi}$
with the symplectic form $p_1^*(k\omega_\lambda)-p_2^*\omega^K_\xi$.
The diagonal action of $K$ on 
$X \times \overline{\mathcal{O}^K_\xi}$, when equipped 
with this symplectic form, is then holomorphic and 
Hamiltonian with moment map 

\begin{equation*}
\mu_\fk^{k,\xi}(x,y):=\mu_\fk^k(x)-y, \quad (x,y) \in 
X \times \overline{\mathcal{O}^K_\xi}.
\end{equation*}

Put 
\begin{equation*}
\mathcal{L}(\xi, k):=
p_1^*\mathscr{L}^k \otimes p_2^*\overline{\mathscr{L}_\xi^*}.
\end{equation*}
Then $\mathcal{L}(\xi, k)$ is a holomorphic 
prequantum line bundle for 
$((X \times \overline{\mathcal{O}^K_\xi})$,
$p_1^*(k\omega_\lambda)-p_2^*\omega^K_\xi)$, 
and the $K$-action lifts to a holomorphic action 
on $\mathcal{L}(\xi, k)$.

From Theorem~\ref{T: reducedspaces} (c) we immediately conclude 
the following result.

\begin{prop} \label{P: transzerofibre}
For every $k \in \N$, the group $K$ acts transitively on 
$(\mu_{\fk}^{k,\xi})^{-1}(0) \subseteq 
X \times \overline{\mathcal{O}^K_\xi}$. 
\end{prop}

\begin{prop} \label{P: multfree}
For every $\xi \in (i\ft)^*$, and $k \in \N$, the dimension of the space
$H^0(X \times \overline{\mathcal{O}^K_\xi}, \mathcal{L}(\xi, k))^K$ 
is at most one.
\end{prop}

\begin{proof}
First of all, by \cite[Thm. 2.18]{S95},
\begin{equation*}
H^0(X \times \overline{\mathcal{O}^K_\xi}, 
\mathcal{L}(\xi, k))^K \cong 
H^0((X \times \overline{\mathcal{O}^K_\xi})_{ss},
\mathcal{L}(\xi, k))^K.
\end{equation*}
By the definition of $(X \times \overline{\mathcal{O}^K_\xi})_{ss}$ 
any $K$-invariant section, being also $K^\CC$-invariant, is 
uniquely determined by its values on $(\mu_\fk^{k,\xi})^{-1}(0)$.
By Proposition \ref{P: transzerofibre}, such a section is in fact 
determined by its value at some given point in 
$(\mu_\fk^{k,\xi})^{-1}(0)$. 
This finishes the proof.
\end{proof}

\begin{lemma}\label{L: integral}
The integral points in 
$\mu_{\fk}^k(X)\cap i\mathfrak{c}$
are precisely the points
$k\lambda+m_1\gamma_1+\cdots+m_r\gamma_r$, where
$k \geq m_1 \geq \ldots \geq m_r \geq 0$. 
\end{lemma}

\begin{proof}
Clearly, every point $k\lambda+m_1\gamma_1+\cdots+m_r\gamma_r$ with 
$k \geq m_1 \geq \ldots \geq m_r \geq 0$ is integral. 
For the converse inclusion, let 
$\xi \in \mu_{\fk}^k(X)\cap i\mathfrak{c}$.
Then $\xi=k\lambda+x_1\gamma_1+\cdots+x_r\gamma_r$ 
for some $x_i \in \R$ satisfying
$k \geq x_1 \geq \ldots \geq x_r \geq 0$. 
If $\xi$ is integral, then $x_1\gamma_1+\cdots+x_r\gamma_r$
is also integral. From the argument at the end of 
the proof of Lemma 2 in \cite{Sch} it 
then follows that $x_i \in \Z$ for $i=1,\ldots, r$. 
This finishes the proof.
\end{proof}

\begin{thm}\label{T: main}
The space $H^0(X, \mathscr{L}^k)$ decomposes under
$L$ as
\begin{equation*}
H^0(X, \mathscr{L})=\bigoplus_{k \geq m_1 \geq \ldots \geq m_r \geq 0}
W_{(k\lambda, \boldsymbol{m})},
\end{equation*}
where $W_{(k\lambda, \boldsymbol{m})}$ is the irreducible 
$L$-representation with highest weight 
$k\lambda+\sum_{i=1}^r m_i\gamma_i$. 
\end{thm}

\begin{proof}
By Theorem~\ref{T: sjamaar} and Lemma~\ref{L: integral}, only representations of the form 
$W_{(k\lambda, \boldsymbol{m})}$ can occur in 
$H^0(X, \mathscr{L})$, and by Proposition \ref{P: multfree} they 
can at most have multiplicity one. Thus, it suffices 
to prove that every such representation actually does occur. 
For this, we note that every $\boldsymbol{m} \in \N_0^r$ 
satisfying the condition 
$k \geq m_1 \geq \ldots \geq m_r \geq 0$ can 
be written uniquely as
$\boldsymbol{m}=\sum_{j=1}^k \boldsymbol{m}(j)$, 
with $\boldsymbol{m}(j)=(m_1(j),\ldots, m_r(j)) 
\in \N_0^r$ satisfying 
$1 \geq m_1(j) \geq \ldots \geq m_r(j) \geq 0$.
By Theorem~\ref{thm:Decomposition} and Proposition~\ref{P: niceterm}, for each such $\boldsymbol{m}(j)$ the 
irreducible $L$-representation $W_{(\lambda, \boldsymbol{m}(j))}$
occurs in $H^0(X, \mathscr{L})$. 
Let $s_j \in H^0(X, \mathscr{L})$ be an $L$-highest 
weight vector for the representation 
$W_{(\lambda, \boldsymbol{m}(j))}$. 
Then $s_1 \cdots s_k \in H^0(X, \mathscr{L}^k)$ is 
an $L$-highest weight vector of weight 
$k\lambda+\sum_{i=1}^rm_i\gamma_i$. Hence, 
the representation $W_{(k\lambda, \boldsymbol{m})}$
occurs in $H^0(X, \mathscr{L}^k)$.
\end{proof}

\section{The Okounkov body}
In this section we return to the Okounkov body 
$\Delta(\mathscr{L}, N^+_L, v)$. The main result 
is an identification of $\Delta(\mathscr{L}, N^+_L, v)$ 
with the moment polytope $\Pi_{\mathfrak{s}}$. We also 
prove that the semigroup $S(\mathscr{L}, N^+_L, v)$ 
is finitely generated. 

Let $\boldsymbol{m} \in \N_0^r$ satisfy
$1 \geq m_1 \geq \ldots \geq m_r \geq 0$, and let
$s_{\boldsymbol{m}} \in H^0(X,\mathscr{L})$ be an 
$L$-highest weight vector with highest weight 
$\lambda+m_1\gamma_1+\cdots+m_r\gamma_r$.

\begin{prop} \label{P: finsemi}
The semigroup $S(\mathscr{L}, N^+_L, v)$ is generated
by the elements $(1,v(s_{\boldsymbol{m}}))$ with 
$1 \geq m_1 \geq \ldots \geq m_r \geq 0$.
\end{prop}

\begin{proof}
Let $s \in H^0(X,\mathscr{L}^k)$ be an 
$N^+_L$-invariant vector. Then $s$ can be 
written as a linear combination $s=\sum_i s_i$ of $L$-highest
weight vectors in $H^0(X,\mathscr{L}^k)$ corresponding 
to distinct highest weights. These weights are
distinct also as $\fh$-weights. Hence, by Proposition 
\ref{P: monomweight} $v(s_i) \neq v(s_j)$ for $i \neq j$.
It follows immediately from the definition of $v$ 
that $v(s)=v(s_i)$ for some $s_i$, namely 
the $s_i$ with smallest $v(s_i)$.
We can thus without loss of generality assume that 
$s$ is an $L$-highest weight vector. 
By Theorem \ref{T: main}, the weight of 
$s$ can then be written as $k\lambda+\sum_{i=1}^rm_i\gamma_i$
with $k \geq m_1 \geq \ldots \geq m_r \geq 0$. 
As in the proof of Theorem \ref{T: main}, 
$\boldsymbol{m}:=(m_1,\ldots, m_r)$ can be 
written uniquely as 
$\boldsymbol{m}=\sum_{j=1}^k \boldsymbol{m}(j)$, 
with $\boldsymbol{m}(j)=(m_1(j),\ldots, m_r(j)) 
\in \N_0^r$ satisfying 
$1 \geq m_1(j) \geq \ldots \geq m_r(j) \geq 0$. 
Let $s_j \in H^0(X, \mathscr{L})$ be an $L$-highest 
weight vector with weight 
$\lambda+\sum_{i=1}^r m_i(j)\gamma_i$.
Then $s_1 \cdots s_k \in H^0(X, \mathscr{L}^k)$ is 
an $L$-highest weight vector of weight 
$k\lambda+\sum_{i=1}^rm_i\gamma_i$. Since 
the decomposition under $L$ is multiplicity free, 
$s$ is a scalar multiple of $s_1\cdots s_k$. 
Hence, $v(s)=v(s_1)+\cdots+v(s_k)$. This finishes the 
proof.
\end{proof}

\begin{cor}
The Okounkov body $\Delta(\mathscr{L}, N^+_L, v)$ is 
the convex hull of the points
$(1,v(s_{\boldsymbol{m}}))$ with 
$1 \geq m_1 \geq \ldots \geq m_r \geq 0$.
\end{cor}

Let $\mathcal{P} \subseteq \ft^*$ denote the weight 
lattice of $\ft$.
Define the \emph{moment semigroup} 
\begin{align*}
S(\mathscr{L}, N^+_L):=\{(k, i\xi) \in \N \times \mathcal{P} \mid \xi \;
\mbox{is the highest weight of some} 
\,W_{(k\lambda, \boldsymbol{m})}\}.
\end{align*}
We now construct a morphism of semigroups, 
$\Lambda$, from $S(\mathscr{L}, N^+_L, v)$ to $S(\mathscr{L}, N^+_L)$.
As already observed in the proof of Proposition~\ref{P: finsemi}, 
two weight vectors $s_1, s_2 \in H^0(X,\mathscr{L}^k)$ with 
distinct weights cannot have the same value under $v$ .
Now, let $(k,a) \in S(\mathscr{L}, N^+_L,v)$. 
Since the decomposition of $H^0(X,\mathscr{L}^k)$ is multiplicity
free there exists a unique (up to scalar multiples) weight vector 
$s \in H^0(X,\mathscr{L}^k)^{N^+_L}$ with $v(s)=a$. 
We therefore define
\begin{align*}
\Lambda: S(\mathscr{L}, N^+_L, v) \rightarrow \mathcal{P}
\end{align*}
by putting $\Lambda((k,a))=(k,i\xi)$, where $\xi$ is the weight 
of $s$.

\begin{rem}
The morphism $\Lambda$ was introduced by Okounkov in \cite{Ok}
in a slightly different (but more general) setting.
\end{rem}

\begin{prop} \label{P: isosemigroup}
The map $\Lambda: S(\mathscr{L}, N^+_L, v) \rightarrow 
S(\mathscr{L}, N^+_L)$ is an isomorphism of semigroups.
\end{prop}

\begin{proof}
We first prove that $\Lambda$ is injective. 
For this, assume that for some $a \neq b$ we have $\Lambda((k,a))=\Lambda((k,b))$. If $s, t \in H^0(X, \mathscr{L}^k)^{N^+_L}$ are 
highest weight vectors with $v(s)=a$ and $v(t)=b$, respectively, 
then $s$ and $t$ are linearly independent highest weight vectors 
of the same weight. This would, however, contradict the 
multiplicity-freeness of the decomposition of $H^0(X, \mathscr{L}^k)$
under $L$. 

For the surjectivity, we note that, again by multiplicity-freeness, 
every highest weight $\xi$ of some $W_{(k\lambda, \boldsymbol{m})}$ 
is the weight of a unique (up to scalar multiples) 
highest weight vector $s \in H^0(X, \mathscr{L}^k)^{N^+_L}$.
Then $\Lambda((k, v(s)))=i\xi$.
\end{proof}
Let $E(\mathscr{L}, N^+_L, v) \subseteq \R \times \R^n$ be 
the $\R$-linear subspace generated by $S(\mathscr{L}, N^+_L, v)$, 
and let $E(\mathscr{L}, N^+_L) \subseteq \R \times \ft^*$ be 
the $\R$-linear subspace generated by $S(\mathscr{L}, N^+_L)$.
We extend $\Lambda$ to a unique linear map
$E(\mathscr{L}, N^+_L, v)\rightarrow E(\mathscr{L}, N^+_L) $, 
and we let $\Lambda$ also denote this extension.

Let $C(\mathscr{L}, N^+_L) \subseteq E(\mathscr{L}, N^+_L)$ 
be the closed convex cone generated by $S(\mathscr{L}, N^+_L)$, and put
\begin{equation*}
\Delta(\mathscr{L}, N^+_L):=C(\mathscr{L}, N^+_L) \cap (\{1\} \times \ft^*).
\end{equation*}
Clearly, we can identify $\Delta(\mathscr{L}, N^+_L)$ 
with $\Pi_{\mathfrak{s}}$ by
\begin{equation*}
\Delta(\mathscr{L}, N^+_L)=\{(1,x) \mid x \in \Pi_{\mathfrak{s}}\}.
\end{equation*}
Using this identification, Proposition \ref{P: isosemigroup} 
readily yields the following identification of the Okounkov 
body $\Delta(\mathscr{L}, N^+_L, v)$ with the moment polytope
$\Pi_{\mathfrak{s}}$.
\begin{thm}
The map 
$\Lambda: E(\mathscr{L}, N^+_L, v)\rightarrow E(\mathscr{L}, N^+_L)$
restricts to a bijection
\begin{equation*}
\Delta(\mathscr{L}, N^+_L, v) \rightarrow \Pi_{\mathfrak{s}}.
\end{equation*}

\begin{rem}
The problem of constructing polyhedral Okounkov bodies has 
been addressed both in the setting of group actions, such in 
\cite{Ok}, as well as in the case, developed in \cite{LM09}, when 
the semigroup is defined by the values of all sections.
Few positive results in this direction are known, however. 
It is known to work for torus-equivariant line bundles over 
toric varieties (cf. \cite{LM09}). As examples in the setting of 
homogeneous spaces under a reductive group we would like to 
mention \cite{Ok2} and, more recently, \cite{kcrystal}.
\end{rem}

\end{thm}

\end{document}